\documentclass[11pt]{article}

\usepackage{amsmath,amssymb,verbatim,lpic}
\usepackage{amsthm}
\usepackage{stmaryrd}
\usepackage{enumerate,multicol}
\usepackage{url}

\usepackage[margin=1in]{geometry}

\AtBeginDocument{%
   \def\MR#1{}
}

\newtheorem{theorem}{Theorem}[section]
\newtheorem{proposition}[theorem]{Proposition}
\newtheorem{lemma}[theorem]{Lemma}
\newtheorem{corollary}[theorem]{Corollary}
\newtheorem{alphatheorem}{Theorem}

\theoremstyle{definition}
\newtheorem{definition}[theorem]{Definition}
\newtheorem{fact}[theorem]{Fact}
\newtheorem{conjecture}[theorem]{Conjecture}

\newtheorem{remark}[theorem]{Remark}
\newtheorem{example}[theorem]{Example}

\def\M{\mathbb M}
\def\N{\mathbb N}

\def\Q{\mathbb Q}
\def\R{\mathbb R}
\def\U{\mathbb U}

\newcommand{\cG}{\mathcal{G}}

\newcommand{\cI}{\mathcal{I}}
\newcommand{\cJ}{\mathcal{J}}
\newcommand{\cK}{\mathcal{K}}
\newcommand{\cL}{\mathcal{L}}
\newcommand{\caL}{\mathcal{L}}
\newcommand{\cM}{\mathcal{M}}
\newcommand{\cN}{\mathcal{N}}

\newcommand{\cQ}{\mathcal{Q}}
\newcommand{\cR}{\mathcal{R}}
\newcommand{\cS}{\mathcal{S}}
\newcommand{\cU}{\mathcal{U}}

\newcommand{\p}{\oplus}
\newcommand{\m}{\ominus}
\newcommand{\Fraisse}{Fra\"{i}ss\'{e}}

\newcommand{\LDS}{\cL_{\textnormal{om}}}

\newcommand{\TR}[1]{\Th(\cU_{#1})}
\newcommand{\UR}[1]{\cU_{#1}}
\newcommand{\MU}[1]{\U_{#1}}
\newcommand{\oms}[1]{\vphi_{\textnormal{#1}}}
\newcommand{\RUS}{\mathbf{RUS}}

\newcommand{\vphi}{\varphi}
\newcommand{\func}{\longrightarrow}

\newcommand{\seq}{\subseteq}

\newcommand{\mand}{\makebox[.4in]{and}}
\newcommand{\miff}{\makebox[.4in]{$\Leftrightarrow$}}
\newcommand{\uth}{^{\textrm{th}}}
\newcommand{\abar}{\bar{a}}
\newcommand{\bbar}{\bar{b}}
\newcommand{\cbar}{\bar{c}}

\newcommand{\xbar}{\bar{x}}
\newcommand{\ybar}{\bar{y}}
\newcommand{\zbar}{\bar{z}}
\newcommand{\albar}{\bar{\alpha}}

\def\Th{\operatorname{Th}}
\def\SO{\operatorname{SO}}

\def\SOP{\operatorname{SOP}}
\def\TP{\operatorname{TP}}
\def\IP{\operatorname{IP}}
\def\NSOP{\operatorname{NSOP}}
\def\FSOP{\operatorname{FSOP}}
\def\NFSOP{\operatorname{NFSOP}}

\def\NTP{\operatorname{NTP}}
\def\NIP{\operatorname{NIP}}
\def\arch{\operatorname{arch}}
\def\dcl{\operatorname{dcl}}
\def\acl{\operatorname{acl}}
\def\Aut{\operatorname{Aut}}
\def\Sym{\operatorname{Sym}}
\def\tp{\operatorname{tp}}
\def\NP{\operatorname{NP}}
\def\dist{\operatorname{dist}}
\def\dom{\operatorname{dom}}

\def\eacl{\operatorname{acl}^{\operatorname{eq}}}
\def\edcl{\operatorname{dcl}^{\operatorname{eq}}}
\def\eq{\operatorname{eq}}
\def\heq{\operatorname{heq}}

\def\Ind{\setbox0=\hbox{$x$}\kern\wd0\hbox to 0pt{\hss$\mid$\hss}
\lower.9\ht0\hbox to 0pt{\hss$\smile$\hss}\kern\wd0}

\def\Notind{\setbox0=\hbox{$x$}\kern\wd0\hbox to 0pt{\mathchardef
\nn=12854\hss$\nn$\kern1.4\wd0\hss}\hbox to
0pt{\hss$\mid$\hss}\lower.9\ht0 \hbox to 0pt{\hss$\smile$\hss}\kern\wd0}

\def\ind{\mathop{\mathpalette\Ind{}}}

\def\nind{\mathop{\mathpalette\Notind{}}}

\newcommand{\dotminus}{ 
\!\!\buildrel\textstyle~.\over{\hbox{ 
\vrule height3pt depth0pt width0pt}{\smash-} 
}}

\newcommand{\odotminus}{ 
\!\!\buildrel\textstyle~.\over{\hbox{ 
\vrule height1.5pt depth0pt width0pt}{\smash\ominus} 
}}

\begin{document}

\title{Neostability in countable homogeneous metric spaces}

\author{
Gabriel Conant\\
University of Notre Dame\\
gconant@nd.edu
}

\date{September 6, 2016}

\maketitle

\begin{abstract}
Given a countable, totally ordered commutative monoid $\cR=(R,\p,\leq,0)$, with least element $0$, there is a countable, universal and ultrahomogeneous metric space $\cU_\cR$ with distances in $\cR$. We refer to this space as the \textit{$\cR$-Urysohn space}, and consider the theory of $\cU_\cR$ in a binary relational language of distance inequalities. This setting encompasses many classical structures of varying model theoretic complexity, including the rational Urysohn space, the free $n\uth$ roots of the complete graph (e.g. the random graph when $n=2$), and theories of refining equivalence relations (viewed as ultrametric spaces). We characterize model theoretic properties of $\Th(\cU_\cR)$ by algebraic properties of $\cR$, many of which are first-order in the language of ordered monoids. This includes stability, simplicity, and Shelah's $\SOP_n$-hierarchy.  Using the submonoid of idempotents in $\cR$, we also characterize superstability, supersimplicity, and weak elimination of imaginaries. Finally, we give necessary conditions for elimination of hyperimaginaries, which further develops previous work of Casanovas and Wagner.
\end{abstract}

\section{Introduction}

\numberwithin{figure}{section}

In this paper, we consider model theoretic properties of generalized metric spaces obtained as analogs of the rational Urysohn space. Our results will show that this class of metric spaces exhibits a rich spectrum of complexity in the classification of first-order theories without the strict order property. 

The object of focus is the countable \textit{$\cR$-Urysohn space}, denoted $\cU_\cR$, where $\cR=(R,\p,\leq,0)$ is a countable totally ordered commutative monoid with least element $0$, or \textit{distance monoid} (see Definition \ref{def:AMS}). We refer to generalized metric spaces taking distances in $\cR$ as \textit{$\cR$-metric spaces}. The space $\cU_\cR$ is then defined to be the unique countable, ultrahomogeneous $\cR$-metric space, which is universal for finite $\cR$-metric spaces. Explicitly, $\cU_\cR$ is the \Fraisse\ limit of the class of finite $\cR$-metric spaces, and its existence follows from the work in the prequel \cite{CoDM1} to this paper, which generalizes previous results of Delhomm\'{e}, Laflamme, Pouzet, and Sauer \cite{DLPS}. In particular, associativity of $\p$ is not required to define $\cR$-metric spaces and, in fact, characterizes when the class of finite $\cR$-metric spaces is a \Fraisse\ class (see \cite[Theorem 5]{Sa13b} or \cite[Proposition 5.7]{CoDM1}). We give a few examples of historical and mathematical significance.

\begin{example}\label{allEX}
$~$
\begin{enumerate}
\item Let $\cQ=(\Q^{\geq0},+,\leq,0)$ and $\cQ_1=(\Q\cap[0,1],+_1,\leq,0)$, where $+_1$ is addition truncated at $1$. Then $\cU_\cQ$ and $\cU_{\cQ_1}$ are, respectively, the \textit{rational Urysohn space} and \textit{rational Ursyohn sphere}. The completion of $\cU_\cQ$ is called the \textit{Urysohn space}, and is the unique complete, separable metric space, which is homogeneous and universal for separable metric spaces. The completion of $\cU_{\cQ_1}$ is called the \textit{Urysohn sphere} and satisfies the same properties with respect to separable metric spaces of diameter $\leq 1$. These spaces were originally constructed by Urysohn in 1925 (see \cite{Ury}, \cite{Ury2}). 

\item Let $\cR_2=(\{0,1,2\},+_2,\leq,0)$, where $+_2$ is addition truncated at $2$. Then $\cU_{\cR_2}$ is isometric to the \textit{countable random graph} or \textit{Rado graph} (when equipped with the path metric). A directed version of this graph was first constructed by Ackermann in 1937 \cite{AckerRG}. The undirected construction is usually attributed to Erd\H{o}s and R\'{e}nyi \cite[1963]{ErdRen} or Rado \cite[1964]{Rado}. 
\item Generalize the previous example as follows. Fix $n>0$ and let $\cR_n=(\{0,1,\ldots,n\},+_n,\leq,0)$, where $+_n$ is addition truncated at $n$. Let $\cN=(\N,+,\leq,0)$. We refer to $\cU_{\cR_n}$ as the \textit{integral Urysohn space of diameter $n$}, and to $\cU_\cN$ as the \textit{integral Urysohn space}. These spaces were constructed by Pouzet and Roux \cite[1996]{PoRo} and Cameron \cite[1998]{CaDTG}.  Also, Casanovas and Wagner \cite[2004]{CaWa} construct the \textit{free $n\uth$ root of the complete graph}. As with $n=2$, equipping this graph with the path metric yields $\cU_{\cR_n}$. 
\item Generalize all of the previous examples as follows. Fix a countable subset $S\seq\R^{\geq0}$ closed under the operation $r+_S s:=\sup\{x\in S:x\leq r+s\}$ and containing $0$. Let $\cS=(S,+_S,\leq,0)$ and assume $+_S$ associative. Then we have the $\cS$-Urysohn space $\cU_\cS$. This situation is studied in further generality by Delhomm\'{e}, Laflamme, Pouzet, and Sauer \cite{DLPS}.

\item For an example of a different flavor, fix a countable linear order $(R,\leq,0)$, with least element $0$, and let $\cR=(R,\max,\leq,0)$. We refer to $\cU_\cR$ as the \textit{ultrametric Urysohn space over $(R,\leq,0)$}. Explicit constructions of these spaces are given by Gao and Shao in \cite{GaCh}. Alternatively, $\cU_\cR$ is a countable model of the theory of infinitely refining equivalence relations indexed by $(R,\leq)$. These are standard model theoretic examples, often used to illustrate various behavior in the stability spectrum (see \cite[Section III.4]{Babook}). 
\end{enumerate}
\end{example}

We will consider model theoretic properties of $\cR$-Urysohn spaces. In particular, given a countable distance monoid $\cR$, we let $\TR{\cR}$ be the complete $\cL_R$-theory of $\UR{\cR}$, where $\cL_R$ is a first-order language consisting of binary relations $d(x,y)\leq r$, for $r\in R$. In \cite{CoDM1}, we constructed a ``nonstandard" distance monoid extension $\cR^*$ of $\cR$, with the property that any model of $\TR{\cR}$ is canonically an $\cR^*$-metric space (see Theorem \ref{thm:R*} below). We let $\MU{\cR}$ denote a sufficiently saturated monster model of $\TR{\cR}$. Then, as an $\cR^*$-metric space, $\MU{\cR}$ is $\kappa^+$-universal, where $\kappa$ is the saturation cardinal (see \cite[Proposition 6.1]{CoDM1}). However, in order to conclude that $\U_\cR$ is also $\kappa$-homogeneous as an $\cR^*$-metric space, we must assume quantifier elimination. Therefore we say that $\cR$ is a \textbf{Urysohn monoid} if it is a countable distance monoid and $\TR{\cR}$ has quantifier elimination. In \cite[Theorem  6.10]{CoDM1}, we characterized quantifier elimination via continuity of addition in $\cR^*$ (see Theorem \ref{thm:QE} below). This motivates a general schematic for analyzing the model theoretic behavior of $\TR{\cR}$. 

\begin{definition}
Let $\RUS$ denote the class of $\cR$-Urysohn spaces $\cU_\cR$, where $\cR$ is a Urysohn monoid. We say a property $P$ of $\RUS$ is \textbf{axiomatizable} (resp. \textbf{finitely axiomatizable}) if there is an $\cL_{\omega_1,\omega}$-sentence (resp. $\cL_{\omega,\omega}$-sentence) $\vphi_P$, in the language of ordered monoids, such that, if $\cR$ is a Urysohn monoid, then $\cU_\cR$ satisfies $P$ if and only if $\cR\models\vphi_P$.
\end{definition} 

Regarding this definition, it is worth mentioning that there is a first-order sentence $\oms{QE}$, in the language of ordered monoids, such that a countable distance monoid $\cR$ is a Urysohn monoid if and only if $\cR\models\oms{QE}$ (see \cite[Corollary 6.11]{CoDM1}). Therefore, if some property $P$ is axiomatizable with respect to the class of \textit{all} $\cR$-Urysohn spaces, then $P$ is also axiomatizable relative to $\RUS$. 

Our results in this direction begin with notions around stability and simplicity. In particular, the ultrametric spaces in Example \ref{allEX}(5) are well-known to be stable when considered as refining equivalence relations. Furthermore, the countable random graph (Example \ref{allEX}(2)) is a canonical example of a simple unstable structure. In Section \ref{sec:low}, we use a general characterization of forking and dividing in $\TR{\cR}$, in conjunction with several other ternary relations on $\U_\cR$, to prove the following theorem, which summarizes Theorems \ref{thm:stable}, \ref{thm:simple}, \ref{thm:supsim}, and \ref{thm:supstab}.  

\begin{alphatheorem}\label{prethm:SS}$~$
\begin{enumerate}[$(a)$]
\item Stability and simplicity are finitely axiomatizable properties of $\RUS$. In particular, given a Urysohn monoid $\cR$,
\begin{enumerate}[$(i)$]
\item $\TR{\cR}$ is stable if and only if $\cR$ is ultrametric, i.e., for all $r,s\in R$, $r\p s=\max\{r,s\}$;
\item $\TR{\cR}$ is simple if and only if, for all $r\leq s$ in $R$, $r\p r\p s=r\p s$.
\end{enumerate}
\item Suppose $\cR$ is a Urysohn monoid. Then $\TR{\cR}$ is supersimple if and only if $\TR{\cR}$ is simple and the submonoid of idempotent elements of $\cR$ is well-ordered. Moreover, in this case $SU(\TR{\cR})$ is the order type of the set of non-maximal idempotents in $\cR$. Applying part $(a)(i)$, $\TR{\cR}$ is superstable if and only if $\TR{\cR}$ is stable and $\cR$ is well-ordered.
\end{enumerate}
\end{alphatheorem}

 This result fits nicely into results of Koponen \cite{Kop} on simple homogeneous structures in finite binary relational languages. In particular, for finite $\cR$, it follows from Koponen's work that, if $\TR{\cR}$ is simple, then it is supersimple with $SU$-rank bounded by $|R|$. Therefore, Theorem \ref{prethm:SS} provides a sharper bound on the $SU$-rank in the supersimple case. Part $(b)$ also implies that superstability and supersimplicity are not axiomatizable in $\RUS$ (see Corollary \ref{cor:undef}).
 
Having established the presence of generalized Urysohn spaces in the most well-behaved regions of classification theory, we then turn to the question of how complicated $\TR{\cR}$ can be. For example, Theorem \ref{prethm:SS} immediately implies that the rational Urysohn space (Example \ref{allEX}(1)) is not simple. This is a well-known fact, which was observed for the complete Urysohn sphere in continuous logic by Pillay (see \cite{EaGo}). Casanovas and Wagner give a similar argument in \cite{CaWa} to show that $\TR{\cQ_1}$ is not simple. Regarding an upper bound in complexity, it is shown in \cite{CoTe} (joint work with Caroline Terry) that the complete Urysohn sphere does not have the \textit{fully finitary} strong order property. Altogether, this work sets the stage for the main result of Section \ref{sec:NFSOP}, which gives the following upper bound for the complexity of $\TR{\cR}$ (see Corollary \ref{cor:NFSOP}).

\begin{alphatheorem}
If $\cR$ is a Urysohn monoid then $\TR{\cR}$ does not have the finitary strong order property.
\end{alphatheorem}

In Section \ref{sec:SOR}, we address the region of complexity between simplicity and the finitary strong order property, which, in general, is stratified by Shelah's $\SOP_n$-hierarchy. Concerning $\TR{\cR}$, we first use the characterizations of stability and simplicity to formulate a purely algebraic notion of the \textit{archimedean complexity}, $\arch(\cR)$, of a general distance monoid $\cR$ (see Definition \ref{def:arch}). In particular, $\TR{\cR}$ is stable (resp. simple) if and only if $\arch(\cR)\leq 1$ (resp. $\arch(\cR)\leq 2$). We then use this rank to pinpoint the exact complexity of $\TR{\cR}$.

\begin{alphatheorem}
If $\cR$ is a Urysohn monoid and $n\geq 3$, then $\TR{\cR}$ is $\SOP_n$ if and only if $\arch(\cR)\geq n$.
\end{alphatheorem}

In particular, this seems to give the first natural class of structures, in which the $\SOP_n$-hierarchy is uniformly represented by meaningful structural behavior independent of combinatorial dividing lines.  

Finally, in Section \ref{sec:EHI}, we consider the question of \textit{elimination of hyperimaginaries}. Hyperimaginaries were introduced in \cite{HKP} in order to define canonical bases in simple theories, which naturally led to the question of when they can be ``eliminated" down to imaginaries (see Definition \ref{def:EHI}). It is known that stable and supersimple theories eliminate hyperimaginaries (see, e.g., \cite{Cabook}), and a major open problem is whether the same is true for all simple theories. In a search for counterexamples, Casanovas and Wagner \cite{CaWa} found the first example of a (non-simple) theory without the strict order property that does not eliminate hyperimaginaries. In particular, they showed that $\TR{\cQ_1}$ is such a theory (although they did not identify their theory as such, see \cite[Proposition 7.5]{CoDM1}). We adapt their methods to give necessary conditions for elimination of hyperimaginaries for $\TR{\cR}$, where $\cR$ is any Urysohn monoid. In proving this, we also characterize weak elimination of imaginaries (see Theorem \ref{thm:WEI}).

\begin{alphatheorem}
If $\cR$ is a Urysohn monoid then $\TR{\cR}$ has weak elimination of imaginaries if and only if $\cR$ has no nonzero and non-maximal idempotents.
\end{alphatheorem}

In the case that $\cR$ has finite archimedean complexity (i.e. $\TR{\cR}$ is $\NSOP_n$ for some $n\geq 3$), it follows that $\TR{\cR}$ has weak elimination of imaginaries if and only if $\cR$ is archimedean as an ordered monoid. In \cite{CoStr}, this characterization of weak elimination of imaginaries is used to provide a counterexample to a question of Adler \cite{Adgeo} concerning theories with multiple strict independence relations (see Remark \ref{rem:thorn}).

\section{Preliminaries}\label{sec:prelim}

\subsection{Classification Theory}\label{sec:preCT}

In this section, $T$ denotes a complete first-order theory and $\M$ denotes a sufficiently saturated monster model of $T$. We specify notation and conventions, which will apply throughout the paper. We write $A\subset\M$ to denote that $A$ is a subset of $\M$ and $|A|<|\M|$. We write $\bbar\in\M$ to denote that $\bbar$ is a tuple of elements of $\M$. Tuples may be infinite in length, but always smaller in cardinality than $\M$. The letters $a,b,c,\ldots$ always denote \textit{singletons} in $\M$. We use $\ell(\abar)$ to denote the length, or domain, of the tuple $\abar$. By convention, all indiscernible sequences are over infinite index sets, and \emph{indiscernible} means \emph{order indiscernible}. Given $A,B,C\subset\M$, we use the notation $A\ind^f_C B$ (resp. $A\ind^d_C B$) to denote that $\tp(A/BC)$ does not fork (resp. divide) over $C$. 

We assume the reader is familiar with the basics of forking, dividing, stability, and simplicity. We will use the following fact concerning the behavior of nonforking in simple and stable theories.

\begin{fact}\label{fact:low}\textnormal{\cite[Theorem 2.6.1, Remark 2.6.9]{Wabook}}
$~$
\begin{enumerate}[$(a)$]
\item $T$ is simple if and only if $\ind^f$ satisfies symmetry if and only if $\ind^f$ satisfies local character.
\item $T$ is stable if and only if $\ind^f$ satisfies symmetry and stationarity over models.
\end{enumerate}
\end{fact}

Finally, we discuss Shelah's $\SOP_n$-hierarchy. For the sake of brevity, we present this hierarchy by means of a rank on first-order theories.

\begin{definition}\label{def:cyclic}
Suppose $\cI=(\abar^l)_{l<\omega}$ is an indiscernible sequence in $\M$. 
\begin{enumerate}
\item Given $n>0$, $\cI$ is \textbf{$n$-cyclic} if 
$$
p(\xbar^1,\xbar^2)\cup p(\xbar^2,\xbar^3)\cup\ldots\cup p(\xbar^{n-1},\xbar^n)\cup p(\xbar^n,\xbar^1)
$$
is consistent, where $p(\xbar,\ybar)=\tp(\abar^0,\abar^1)$.
\item Define the set of \textbf{non-parameter indices of $\cI$}, denoted $\NP(\cI)$, to be $\{i\in\ell(\abar^0):a^0_i\neq a^1_i\}$.
\end{enumerate}
\end{definition}

In particular, $\cI$ is $1$-cyclic if and only if $p(\xbar^1,\xbar^1)$ is consistent, which is equivalent to $\cI$ being a constant sequence. The separation of non-parameter indices will allow us to avoid working with indiscernible sequences over parameters.

\begin{definition}\label{def:SOrank}
Let $T$ be a complete first-order theory.
\begin{enumerate}
\item $T$ has the \textbf{finitary strong order property} ($\FSOP$), if there is an indiscernible sequence $\cI$ such that $\NP(\cI)$ is finite and $\cI$ is not $n$-cyclic for any $n>0$. We say $T$ is $\NFSOP$ if it does not have $\FSOP$.
\item Suppose $T$ is $\NFSOP$. We define the \textbf{strong order rank of $T$}, which is denoted $\SO(T)$ and takes values in $\omega+1$, as follows:
\begin{enumerate}[$(i)$]
\item Given $n<\omega$, $\SO(T)\leq n$ if every indiscernible sequence in $\M$ is $(n+1)$-cyclic.
\item $\SO(T)=\omega$ if $\SO(T)> n$ for all $n<\omega$. 
\end{enumerate}
\end{enumerate}
\end{definition}

Given a structure $\cM$, in some language $\cL$, we let $\SO(\cM)=\SO(\Th(\cM))$. We also say that $\SO(T)$ is \textbf{undefined} if $T$ has $\FSOP$. The following fact places this notion of strong order rank in the context of commonly known dividing lines invented by Shelah. The proofs all involve standard uses Ramsey's theorem to obtain indiscernible sequences. We refer the reader to \cite{Sh500} for the definitions of the dividing lines mentioned in the following fact. Reformulations of these definitions can also be found in \cite{Admock}. A full exposition of this treatment of the $\SOP_n$-hierarchy can be found in the author's thesis \cite[Section 1.4]{Cothesis}.

\begin{fact}\label{fact:SOP}
Let $T$ be a complete first-order theory.
\begin{enumerate}[$(a)$]
\item If $T$ has the strict order property ($\SOP$) then $\SO(T)$ is undefined.
\item Assume $\SO(T)$ is defined (i.e. $T$ is $\NFSOP$).
\begin{enumerate}[$(i)$]
\item $T$ has the strong order property ($\SOP_\omega$) if and only if $\SO(T)=\omega$.
\item Given $n\geq 3$, $\SO(T)\geq n$ if and only if $T$ has the $n$-strong order property ($\SOP_n$).
\end{enumerate}
\item If $T$ is simple then $\SO(T)\leq 2$.
\item $T$ is stable if and only if $\SO(T)\leq 1$.
\item $\SO(T)=0$ if and only if $T$ has finite models.
\end{enumerate}
\end{fact}

\subsection{Generalized Metric Spaces}\label{subsec:GMT}
In this section, we define distance monoids and generalized Urysohn spaces. We then briefly summarize the first-order setting for the theories of these structures, as well as the characterization of quantifier elimination from \cite{CoDM1}.

\begin{definition}\label{def:AMS}$~$
\begin{enumerate}
\item Let $\LDS=\{\p,\leq,0\}$ be the language of ordered monoids. An $\LDS$-structure $\cR=(R,\p,\leq,0)$ is a \textbf{distance monoid} if $(R,\p,0)$ is a commutative monoid and $\leq$ is a total, translation-invariant order with least element $0$.

\item Suppose $\cR$ is a distance monoid. Given a set $A$ and a symmetric function $d:A\times A\func R$, we call $d$ an \textbf{$\cR$-metric} on $A$ if
\begin{enumerate}[$(i)$]
\item for all $x,y\in A$, $d(x,y)=0$ if and only if $x=y$;
\item for all $x,y,z\in A$, $d(x,z)\leq d(x,y)\p d(y,z)$.
\end{enumerate}
In this case, $(A,d)$ is an \textbf{$\cR$-metric space}.
\item Given a countable distance monoid $\cR$, let $\cK_\cR$ denote the class of finite $\cR$-metric spaces, and let $\cU_\cR$ denote the unique (up to isomorphism) countable \Fraisse\ limit of $\cK_\cR$. We call $\cU_\cR$ the \textbf{$\cR$-Urysohn space}. $\cU_\cR$ is the unique (up to isometry) $\cR$-metric space, which is ultrahomogeneous and universal for finite $\cR$-metric spaces.
\end{enumerate}
\end{definition}

The reader may verify that, given a countable distance monoid $\cR$, $\cK_\cR$ is indeed a \Fraisse\ class. The only nontrivial verification is the amalgamation property, and one may simply use the natural generalization of free amalgamation of metric spaces (see \cite[Definition 5.10]{CoDM1}).

\begin{definition}
Suppose $\cR$ is a countable distance monoid. 
\begin{enumerate}
\item Let $\cL_R=\{d(x,y)\leq r:r\in R\}$, where $d(x,y)\leq r$ is a binary relation. We interpret $\cR$-metric spaces as $\cL_R$-structures in the obvious way.
\item Let $\TR{\cR}$ denote the complete $\cL_R$-theory of $\UR{\cR}$. Let $\MU{\cR}$ be a sufficiently saturated monster model of $\UR{\cR}$.
\end{enumerate}
\end{definition}

If $\cR$ is infinite then it is easy to see that saturated models of $\TR{\cR}$ are not $\cR$-metric spaces in any reasonable sense. For example, with $\cQ=(\Q^{\geq0},+,\leq,0)$, $\U_{\cQ}$ exhibits ``distances" that can be naturally interpreted as infinite, infinitesimal, or irrational. However, in \cite{CoDM1}, we proved the following result.

\begin{theorem}\label{thm:R*}\textnormal{\cite{CoDM1}}
Suppose $\cR$ is a countable distance monoid. Then there is a distance monoid extension $\cR^*=(R^*,\p,\leq,0)$ of $\cR$ such that:
\begin{enumerate}[$(a)$]
\item Given $M\models\TR{\cR}$ and $a,b\in M$, there is a unique $\alpha=\alpha(a,b)\in R^*$ such that, for all $r\in R$, $M\models d(a,b)\leq r$ if and only if $\alpha\leq r$. Moreover, if $d_M:M\times M\func R^*$ is such that $d_M(a,b)=\alpha(a,b)$, then $(M,d_M)$ is an $\cR^*$-metric space.
\item $\cR^*$ satisfies the following analytic properties.
\begin{enumerate}[$(i)$]
\item $(R^*,\leq)$ is a Dedekind complete linear order with a maximal element (which coincides with the maximal element of $(R,\leq)$ if one exists).
\item Every non-maximal $r\in R$ has an immediate successor in $R^*$.
\item For all $\alpha,\beta\in R^*$, $\alpha\p\beta=\inf\{r\p s:r,s\in R,~\alpha\leq r,~\beta\leq s\}$.
\end{enumerate}
\end{enumerate}
\end{theorem}

We refer the reader to \cite{CoDM1} for explicit descriptions of $\cR^*$ and its construction. Essentially, $R^*$ corresponds to the set of quantifier-free $2$-types consistent with $\TR{\cR}$. In particular, as a set, $R^*$ depends only on $(R,\leq,0)$. Given a model $M\models\TR{\cR}$, $d_M(a,b)$ is simply the quantifier-free $2$-type realized by $(a,b)$. The operation $\p$ is extended to $\cR^*$ by defining $\alpha\p\beta$ to be the largest $\gamma\in R^*$ such that the $3$-type defining a triangle with distances $\alpha$, $\beta$, and $\gamma$ is consistent.

For example, if $\cR$ is finite then $\cR^*=\cR$. If we again consider $\cQ$, then $(\Q^{\geq0})^*$ can be identified with $\R^{\geq0}\cup\{q^+:q\in\Q^{\geq0}\}\cup\{\infty\}$, where $q^+$ is a new element immediately succeeding $q$, and $\infty$ is an infinite element. In particular, $\cR^*$ is generally \textit{not} an elementary extension of $\cR$ (viewed as $\LDS$-structures). 

A main result in \cite{CoDM1} is the following characterization of quantifier elimination for $\TR{\cR}$. 

\begin{theorem}\label{thm:QE}\textnormal{\cite[Theorem 6.10]{CoDM1}}
Suppose $\cR$ is a countable distance monoid. Then $\TR{\cR}$ has quantifier elimination if and only if, for all $r\in R$, the map $x\mapsto x\p r$ is continuous from $R^*$ to $R^*$.
\end{theorem}

It is worth elaborating a little more here. In particular, Theorem \ref{thm:R*}$(b)(iii)$ already suggests a certain level of continuity in $\cR^*$. In fact, this property is equivalent to the statement that, for all $\alpha\in R^*$, $x\mapsto x\p\alpha$ is \emph{upper semicontinuous} (see \cite[Remark 4.10]{CoDM1} for details). Therefore, quantifier elimination for $\TR{\cR}$ is precisely equivalent to lower semicontinuity in the case when $\alpha\in R$.

Recall that we define a \textbf{Urysohn monoid} to be a countable distance monoid $\cR$ such that $\TR{\cR}$ has quantifier elimination. In \cite[Section 7]{CoDM1}, we showed that this situation includes most natural examples arising in the literature (e.g. each space in Example \ref{allEX}, except the full generality of (4)). 

To ease notation, we will use $d$ for the $\cR^*$-metric $d_{\MU{\cR}}$ on $\MU{\cR}$ given by Theorem \ref{thm:R*}$(a)$. If we restrict $d$ to $\UR{\cR}$ (considered as an $\cR^*$-subspace of $\MU{\cR}$), then $d$ agrees with the original $\cR$-metric on $\UR{\cR}$, which, in turn, agrees with the $\cR^*$-metric $d_{\UR{\cR}}$ on $\UR{\cR}$ given by Theorem \ref{thm:R*}$(a)$.

\begin{proposition}\label{saturated}\textnormal{\cite{CoDM1}}
If $\cR$ is a Urysohn monoid, then $\U_\cR$ is a $\kappa$-homogeneous and $\kappa^+$-universal $\cR^*$-metric space, where $\kappa$ is the saturation cardinal of $\U_\cR$.
\end{proposition} 

The results listed above are essentially all we will need for the subsequent work. In light of Theorem \ref{thm:R*}$(b)(i)$, we adopt the convention that, when considering $\emptyset$ as a subset of $R^*$, we let $\inf\emptyset=\sup R^*$ and $\sup\emptyset=0$. The previously discussed upper semicontinuity in $\cR^*$, provided by Theorem \ref{thm:R*}$(b)(iii)$, will be essential. In particular, it gives us a well-behaved notion of absolute value of the difference between elements of $R^*$.

\begin{definition} 
Given a countable distance monoid $\cR$ and $\alpha,\beta\in R^*$, define 
$$
|\alpha\m \beta|:=\inf\{x\in R^*:\alpha\leq\beta\p x\text{ and }\beta\leq\alpha\p x\}.
$$
\end{definition}

We can now give several useful formulations of upper semicontinuity in $\cR^*$.

\begin{proposition}
Let $\cR$ be a countable distance monoid.
\begin{enumerate}[$(i)$]
\item For all $\alpha,\beta\in R^*$, $\alpha\p\beta=\inf\{r\p s:r,s\in R,~\alpha\leq r,~\beta\leq s\}$.
\item For all $\alpha,\beta,\gamma\in R^*$, $|\alpha\m\beta|\leq\gamma$ if and only if $\alpha\leq\beta\p\gamma$ and $\beta\leq\alpha\p\gamma$.
\item If $X,Y\seq R^*$ then $\inf(X\p Y)=\inf X\p\inf Y$.
\end{enumerate}
\end{proposition}

Note that part $(i)$ simply restates Theorem \ref{thm:R*}$(b)(iii)$. Parts $(ii)$ and $(iii)$ follow from $(i)$ (with the help of Theorem \ref{thm:R*}$(b)$[$(i)$-$(ii)$]). For smoother exposition, we will just say \textbf{by upper semicontinuity} (\textbf{in $\cR^*$}) when applying any part of the previous proposition. Note also that, for any $\alpha\in R^*$, setting $\beta=0$ in part $(i)$ yields $\alpha=\inf\{r\in R:\alpha\leq r\}$. We will say \textbf{by density of $R$} when applying just this property alone.

\section{Notions of Independence}\label{sec:FD}

In this section, we consider various ternary relations on subsets of $\MU{\cR}$, where $\cR$ is a Urysohn monoid. The first such relations are nonforking and nondividing independence. Toward a characterization of these notions, we define the following distance calculations. 

\begin{definition}\label{def:dmax}
Fix a countable distance monoid $\cR$. Given $C\subset\MU{\cR}$ and $b_1,b_2\in\MU{\cR}$, we define
\begin{align*}
d_{\max}(b_1,b_2/C) &= \inf_{c\in C}(d(b_1,c)\p d(c,b_2))\\
d_{\min}(b_1,b_2/C) &= \max\left\{\sup_{c\in C}|d(b_1,c)\m d(c,b_2)|,\textstyle\frac{1}{3}d(b_1,b_2)\right\},
\end{align*}
where, given $\alpha\in R^*$, $\frac{1}{3}\alpha:=\inf\{x\in R^*:\alpha\leq 3x\}$.
\end{definition}

Note that $d_{\max}(b_1,b_2/C)$ is the distance between $b_1$ and $b_2$ in the free amalgamation of $b_1C$ and $b_2C$ over $C$ (equivalently, the largest possible distance between realizations of $\tp(b_1/C)$ and $\tp(b_2/C)$). On the other hand, the interpretation of $d_{\min}$ is not as straightforward, and has to do with the behavior of indiscernible sequences in $\MU{\cR}$ (see Proposition \ref{prop:dmaxmin}). We use these values to give a combinatorial description of $\ind^d$ and $\ind^f$, which, in particular, shows that forking and dividing are the same for complete types in $\TR{\cR}$. 

\begin{theorem}\label{thm:forks}
Suppose $\cR$ is a Urysohn monoid. Given $A,B,C\subset\MU{\cR}$, $A\ind^d_C B$ if and only if $A\ind^f_C B$ if and only if for all $b_1,b_2\in B$, 
$$
d_{\max}(b_1,b_2/AC)=d_{\max}(b_1,b_2/C)\mand d_{\min}(b_1,b_2/AC)=d_{\min}(b_1,b_2/C).
$$
\end{theorem}

This characterization is identical to the characterization, given in \cite{CoTe}, of forking and dividing for the complete Urysohn sphere as a metric structure in continuous logic. Moreover, the proof of Theorem \ref{thm:forks} is very similar to \cite{CoTe}. The work lies in showing that, with only minor modifications, the methods in \cite{CoTe} can be reformulated and applied in discrete logic to $\TR{\cR}$. We give a brief outline of the proof, and the necessary modifications, in \ref{app:FD}.

The analysis of stability and simplicity will use three more ternary relations. 

\begin{definition}
Let $\cR$ be a countable distance monoid. Given $a\in\MU{\cR}$ and $C\subset\MU{\cR}$, define $d(a,C)=\inf\{d(a,c):c\in C\}$. Given $A,B,C\subset\MU{\cR}$, define
\begin{alignat*}{3}
&\textstyle A\ind^{\dist}_C B &&\miff\text{$d(a,BC)=d(a,C)$ for all $a\in A$;}\\
&\textstyle A\ind^\otimes_C B &&\miff \text{$d(a,b)=d_{\max}(a,b/C)$ for all $a\in A$, $b\in B$;}\\
&\textstyle A\ind^{d_{\max}}_C B &&\miff \text{$d_{\max}(b_1,b_2/AC)=d_{\max}(b_1,b_2/C)$ for all $b_1,b_2\in B$.}
\end{alignat*}
\end{definition}

The relation $\ind^{\dist}$ has obvious significance as a notion of independence in metric spaces; and $\ind^{d_{\max}}$ is a reasonable simplification of the characterization of $\ind^f$ given by Theorem \ref{thm:forks}, which will be important in the analysis of when $\TR{\cR}$ is simple. Finally, the relation $A\ind^\otimes_C B$ asserts that, as $\cR^*$-metric spaces, $ABC$ is isometric to the free amalgamation of $AC$ and $BC$ over $C$. This notion has been previously used in work of Tent and Ziegler \cite{TeZiSIR} on \textit{stationary independence relations}, which are ternary relations (defined on some structure, which, in our case, is the monster model $\M$ of a theory $T$) satisfying invariance, monotonicity, symmetry, (full) transitivity, (full) existence, and stationarity (over all sets). We refer the reader to \cite{TeZiSIR} for details.

\begin{proposition}\label{prop:SIR}
$~$
\begin{enumerate}[$(a)$]
\item If $\M$ is the monster model of a complete theory $T$, and $\ind$ is a stationary independence relation on $\M$, then $\ind$ implies $\ind^f$.
\item Suppose $\cR$ is a Urysohn monoid.
\begin{enumerate}[$(i)$]
\item $\ind^f$ implies $\ind^{d_{\max}}$. 
\item $\ind^\otimes$ is a stationary independence relation on $\MU{\cR}$.
\item $\ind^\otimes$ implies $\ind^{\dist}$.
\item $\ind^{\dist}$ satisfies local character.
\end{enumerate}
\end{enumerate}
\end{proposition}
\begin{proof}
Part $(a)$ is a nice exercise in the style of \cite{Adgeo}. See also \cite[Theorem 4.1]{CoTe}.

Part $(b)$. Claim $(i)$ is immediate from Theorem \ref{thm:forks}.

Claim $(ii)$ is left as an exercise. Transitivity and existence require upper semicontinuity in $\cR^*$. Existence and stationarity require quantifier elimination.

For claim $(iii)$, fix $A,B,C\subset\MU{\cR}$, with $A\ind^{\otimes}_C B$. Given $a\in A$ and $b\in B$, we have
$d(a,C)=\inf_{c\in C}d(a,c)\leq d_{\max}(a,b/C)=d(a,b)$. 
Therefore $d(a,BC)=d(a,C)$, and so $A\ind^{\dist}_C B$.

For claim $(iv)$, fix $A,B\subset\MU{\cR}$. We may assume $A,B\neq\emptyset$. We will show that for all $a\in A$, there is $C_a\seq B$ such that $|C_a|\leq\aleph_0$ and $a\ind^{\dist}_{C_a}B$. Then, setting $C=\bigcup_{a\in A}C_a$, we will have $|C|\leq|A|+\aleph_0$ and $A\ind^{\dist}_C B$. 

Fix $a\in A$. If there is some $b\in B$ such that $d(a,b)=d(a,B)$ then set $C_a=\{b\}$. Otherwise, define $X=\{r\in R:d(a,B)\leq r\}$. If $r\in X$ then, by density of $R$, there is some $b_r\in B$ such that $d(a,b_r)\leq r$. Set $C_a=\{b_r:r\in X\}$. For any $b\in B$, by assumption and density of $R$, there is some $r\in X$ such that $r<d(a,b)$. Then $b_r\in C_a$ and $d(a,b_r)\leq r<d(a,b)$, as desired.
\end{proof}

\section{Urysohn Spaces of Low Complexity}\label{sec:low}

\subsection{Stability}\label{sec:stable}

In this section, we characterize the Urysohn monoids $\cR$ for which $\TR{\cR}$ is stable. We recall the following definition from \cite{CoDM1}.

\begin{definition}
A distance monoid $\cR=(R,\p,\leq,0)$ is \textbf{ultrametric} if, for all $r,s\in R$, $r\p s=\max\{r,s\}$.
\end{definition}

It is easy to verify that countable ultrametric monoids are Urysohn (see \cite[Proposition 7.3]{CoDM1}). The goal of this section is to show that, given a Urysohn monoid $\cR$, $\TR{\cR}$ is stable if and only if $\cR$ is ultrametric. The heart of this fact lies in the observation that ultrametric spaces correspond to refining equivalence relations. In particular, if $(A,d)$ is an ultrametric space, then for any distance $r$, $d(x,y)\leq r$ is an equivalence relation on $A$. Therefore, the result that ultrametric monoids yield stable Urysohn spaces recovers classical results on theories of refining equivalence relations (see \cite[Section III.4]{Babook}). 

We begin with the following easy simplification of $d_{\max}$ for ultrametric Urysohn spaces. Note that if $\cR$ is a countable ultrametric monoid then, by upper semicontinuity, $\cR^*$ is ultrametric as well. 

\begin{proposition}\label{prop:ultdmax}
Let $\cR$ be a countable ultrametric monoid. Fix $C\subset\MU{\cR}$ and $b_1,b_2\in\MU{\cR}$. 
\begin{enumerate}[$(a)$]
\item If $d(b_1,c)\neq d(b_2,c)$ for some $c\in C$ then $d_{\max}(b_1,b_2/C)=d(b_1,b_2)$.
\item If $d(b_1,c)=d(b_2,c)$ for all $c\in C$ then $d_{\max}(b_1,b_2/C)=d(b_1,C)$.
\end{enumerate}
\end{proposition}

Next, we observe interaction between ultrametric monoids and the previously discussed notions of independence.

\begin{lemma}\label{prop:stable}
Suppose $\cR$ is a countable distance monoid.
\begin{enumerate}[$(a)$]
\item If $\cR$ is ultrametric then $\ind^\otimes$ coincides with $\ind^{\dist}$ on $\U_\cR$.
\item If $\ind^{\dist}$ is symmetric on $\U_\cR$ then $\cR$ is ultrametric.
\end{enumerate}
\end{lemma}
\begin{proof}
Part $(a)$. By Proposition \ref{prop:SIR}$(b)(iii)$, we only need to show $\ind^{\dist}$ implies $\ind^{\otimes}$. Fix $A,B,C\subset\U_\cR$, with $A\ind^{\dist}_C B$. We want $d(a,b)=d_{\max}(a,b/C)$ for all $a\in A$ and $b\in B$. By Proposition \ref{prop:ultdmax}, we may assume $d(a,c)=d(b,c)$ for all $c\in C$, and prove $d(a,b)=d(a,C)$. Note that $A\ind^{\dist}_C B$ implies $d(a,BC)=d(a,C)$, and so we have $d(a,b)\geq d(a,C)$. On the other hand, if $d(a,b)>d(a,C)$ then there is $c\in C$ such that $d(a,b)>d(a,c)$. But then $d(b,c)=\max\{d(a,b),d(a,c)\}=d(a,b)>d(a,c)$, which contradicts our assumptions.

Part $(b)$. Fix $r,s\in R$. There are $a,b,c\in\MU{\cR}$ such that $d(a,b)=\max\{r,s\}$, $d(a,c)=\min\{r,s\}$, and $d(b,c)=r\p s$. Then $d(a,b)\geq  d(a,c)$, and so $a\ind^{\dist}_c b$. By symmetry, we have $b\ind^{\dist}_c a$, which means $\max\{r,s\}=d(a,b)\geq d(b,c)=r\p s$. Therefore $r\p s=\max\{r,s\}$, and we have shown that $\cR$ is ultrametric. 
\end{proof}

\begin{theorem}\label{thm:stable}
Given a Urysohn monoid $\cR$, the following are equivalent.
\begin{enumerate}[$(i)$]
\item $\TR{\cR}$ is stable.
\item $\ind^f$ coincides with $\ind^{\dist}$.
\item $\ind^f$ coincides with $\ind^{\otimes}$.
\item $\cR$ is ultrametric, i.e. for all $r,s\in R$, if $r\leq s$ then $r\p s=s$.
\end{enumerate}
\end{theorem}
\begin{proof}
$(iv)\Rightarrow(iii)$: Suppose $\cR$ is ultrametric. By Proposition \ref{prop:SIR}$(b)$, it suffices to show $\ind^f$ implies $\ind^\otimes$. So suppose $A\ind^f_C B$ and fix $a\in A$, $b\in B$. We want to show $d(a,b)=d_{\max}(a,b/C)$. By Theorem \ref{thm:forks}, $A\ind^f_C B$ implies $d(b,C)\leq d(a,b)$. So the desired result follows from Proposition \ref{prop:ultdmax}.

$(iii)\Rightarrow(i)$: If $\ind^f$ coincides with $\ind^\otimes$ then $\ind^f$ satisfies symmetry and stationarity by Proposition \ref{prop:SIR}$(b)(ii)$. Therefore $\TR{\cR}$ is stable by Fact \ref{fact:low}$(b)$. 

$(i)\Rightarrow(iv)$: Suppose $\cR$ is not ultrametric. Then we may fix $r\in R$ such that $r<r\p r$. We show that the formula $\vphi(x_1,x_2,y_1,y_2):= d(x_1,y_2)\leq r$ has the order property. We may define a sequence $(a^l_1,a^l_2)_{l<\omega}$ in $\U_\cR$ such that, given $l<m$, $d(a^m_1,a^l_2)=r\p r$, and all other nonzero distances are $r$. Then $\vphi(a^l_1,a^l_2,a^m_1,a^m_2)$ if and only if $l\leq m$.

$(iv)\Rightarrow(ii)$: Combine $(iv)\Rightarrow(iii)$ with Lemma \ref{prop:stable}$(a)$. 

$(ii)\Rightarrow(iv)$: If $\ind^f$ coincides with $\ind^{\dist}$ then $\ind^{\dist}$ is symmetric by Fact \ref{fact:low}$(a)$ and Proposition \ref{prop:SIR}$(b)(iv)$, and so $\cR$ is ultrametric by Lemma \ref{prop:stable}$(b)$.
\end{proof}

\begin{remark}\label{rem:NIP}
Suppose $\cR$ is a Urysohn monoid. As an exercise, the reader may show that if  $r\in R$ is such that $r<r\p r$, then the formula $d(x,y)\leq r$ in fact has the \textit{independence property}, $\IP$, (see \cite[II \S 4]{Shbook}). Consequently, if $\Th(\cU_\cR)$ is $\NIP$ then it is stable. A classical result of Shelah is that a complete theory is stable if and only if it is $\NIP$ and $\NSOP$, where $\SOP$ is the \textit{strict order property} (see \cite[Theorem II.4.7]{Shbook}). It will follow from the results of Section \ref{sec:NFSOP} that, if $\cR$ is a Urysohn monoid, then $\Th(\cU_\cR)$ is always $\NSOP$. This gives a less direct proof of the conclusion that $\NIP$ implies stability for $\Th(\cU_\cR)$. 
\end{remark}

Looking back at this characterization of stability, it is worth pointing out that $(i)$, $(ii)$, and $(iii)$ could all be obtained from $(iv)$ by showing that, when $\cR$ is ultrametric, both $\ind^{\dist}$ and $\ind^\otimes$ satisfy the axioms characterizing nonforking in stable theories (c.f. \cite[Theorem 2.6.1, Remark 2.9.6]{Wabook}). In this way, the above theorem could be entirely obtained without using the general characterization of nonforking given by Theorem \ref{thm:forks}.

\subsection{Simplicity}\label{sec:simple}

Our next goal is an analogous characterization of simplicity for $\TR{\cR}$, when $\cR$ is a Urysohn monoid. We will obtain similar behavior in the sense that simplicity of $\TR{\cR}$ is detected by nicer characterizations of forking, as well as the dynamics of addition in $\cR$. However, unlike the stable case, the characterization of forking and dividing, given by Theorem \ref{thm:forks}, will be crucial for our results.

We begin by defining the preorder on an arbitrary distance monoid given by archimedean equivalence.

\begin{definition}\label{def:archi}
Suppose $\cR$ is a distance monoid.
\begin{enumerate}
\item Define the relation $\preceq$ on $R$ such that $r\preceq s$ if and only if $r\leq ns$ for some $n>0$.
\item Define the relation $\asymp$ on $R$ such that $r\asymp s$ if and only if $r\preceq s$ and $s\preceq r$.
\item Given $r,s\in R$, write $r\prec s$ if $s\not\preceq r$, i.e. if $nr<s$ for all $n>0$.
\end{enumerate}
\end{definition}

Throughout this section, we will use the fact that, given a countable distance monoid $\cR$, if $b\in\MU{\cR}$ and $C\subset\MU{\cR}$ then, by upper semicontinuity in $\cR^*$, $d_{\max}(b,b/C)=2d(b,C)$. We also note the following useful inequality.

\begin{lemma}\label{lem:dmax}
Given a countable distance monoid $\cR$, if $C\subset\MU{\cR}$ and $b_1,b_2,b_3\in\MU{\cR}$ then $d_{\max}(b_1,b_3/C)\leq d_{\max}(b_1,b_2/C)\p d_{\min}(b_2,b_3/C)$.
\end{lemma}
\begin{proof}
For any $c\in C$, we have
\begin{align*}
d_{\max}(b_1,b_3/C) &\leq d(b_1,c)\p d(b_3,c)\\
 &\leq d(b_1,c)\p d(b_2,c)\p |d(b_2,c)\m d(b_3,c)|\\
  &\leq d(b_1,c)\p d(b_2,c)\p d_{\min}(b_2,b_3/C),
\end{align*}
which proves the result.
\end{proof}

We now focus on the ternary relation $\ind^{d_{\max}}$. When $\cR$ is Urysohn, we have that $\ind^f$ implies $\ind^{d_{\max}}$. Our next result, which is the main technical tool in the analysis of simplicity, characterizes when these two relations coincide.

\begin{proposition}\label{simpleFork}
If $\cR$ is a Urysohn monoid, then the following are equivalent.
\begin{enumerate}[$(i)$]
\item $\ind^f$ coincides with $\ind^{d_{\max}}$.
\item For all $r,s\in R$, if $r\leq s$ then $r\p r\p s= r\p s$.
\end{enumerate}
\end{proposition}
\begin{proof}
$(i)\Rightarrow (ii)$. Suppose $(ii)$ fails, and fix $r,s\in R$, with $r\leq s$ and $r\p s<2r\p s$. Define the space $(X,d)$ such that $X=\{a,b_1,b_2,c_1,c_2\}$ and
\begin{equation*}
\begin{aligned}
 &d(a,b_1) =d(b_1,c_1)=d(b_1,c_2)=r,\\
 &d(b_2,c_1) =d(b_2,c_2)=s, \\
 &d(a,c_1) =d(a,c_2)=d(c_1,c_2)=2r,
 \end{aligned}
 \hspace{1cm}
 \begin{aligned}
 &d(b_1,b_2)=r\p s,\\
 &d(a,b_2) =2r\p s.
 \end{aligned}
 \end{equation*}
Then $(X,d)$ is an $\cR^*$-metric space, and so we may assume $(X,d)$ is a subspace of $\U_\cR$. Setting $C=\{c_1,c_2\}$, we have $a\ind^{d_{\max}}_C b_1b_2$ by construction. So to show the failure of $(i)$, we show $a\nind^f_C b_1b_2$. Note that
$$
\sup_{c\in C}|d(b_1,c)\m d(b_2,c)|=|r\m s|\leq s\mand d(b_1,b_2)=r\p s\leq 3s,
$$
which gives $d_{\min}(b_1,b_2/C)\leq s$. Therefore, since $r\p s<2r\p s$, we have
$$
d_{\min}(b_1,b_2/C)\leq s<|(2r\p s)\m r|=|d(a,b_1)\m d(a,b_2)|.
$$
So $a\nind^f_C b_1b_2$ by Theorem \ref{thm:forks}.

$(ii)\Rightarrow(i)$. Assume $\cR$ satisfies $(ii)$. By upper semicontinuity, the same algebraic property holds for $\cR^*$. In particular, $2\alpha=3\alpha$ for all $\alpha\in R^*$, which then implies $2\alpha=n\alpha$ for all $\alpha\in R^*$ and $n\geq 2$.

In order to prove $(i)$, it suffices by Theorem \ref{thm:forks} to show $\ind^{d_{\max}}$ implies $\ind^f$. So assume $A\nind^f_C B$ and, for a contradiction, suppose $A\ind^{d_{\max}}_C B$. By Theorem \ref{thm:forks}, there are $a\in A$ and $b_1,b_2\in B$ such that $d_{\min}(b_1,b_2/C)<|d(a,b_1)\m d(a,b_2)|$. Without loss of generality, we assume $d(a,b_1)\leq d(a,b_2)$, and so we have
\begin{equation*}
d(a,b_1)\p d_{\min}(b_1,b_2/C)<d(a,b_2).\tag{$\dagger$}
\end{equation*}
If $\alpha:=\frac{1}{3}d(b_1,b_2)\leq d(a,b_1)$ then, by $(\dagger)$, 
$$
d(a,b_1)\p\alpha <d(a,b_2)\leq d(a,b_1)\p d(b_1,b_2)\leq d(a,b_1)\p 3\alpha=d(a,b_1)\p 2\alpha,
$$
which contradicts $(ii)$. Therefore, we may assume $d(a,b_1)<\frac{1}{3}d(b_1,b_2)$. 

Suppose, toward a contradiction, that $d(a,b_1)\asymp d_{\max}(a,b_1/C)$. Since $2d(a,b_1)=nd(a,b_1)$ for all $n\geq 2$, it follows that $d_{\max}(a,b_1/C)\leq 2d(a,b_1)$. Combining this observation with $(\dagger)$ and Lemma \ref{lem:dmax}, we have
\begin{align*}
d(a,b_1)\p d_{\min}(b_1,b_2/C) &< d(a,b_2)\\
 &\leq d_{\max}(a,b_1/C)\p d_{\min}(b_1,b_2/C)\\
 &\leq 2d(a,b_1)\p d_{\min}(b_1,b_2/C),
\end{align*}
which, since $d(a,b_1)<\frac{1}{3}d(b_1,b_2)\leq d_{\min}(b_1,b_2/C)$,  contradicts $(ii)$. 

So we have $d(a,b_1)\prec d_{\max}(a,b_1/C)$. Moreover, by Lemma \ref{lem:dmax}, we also have $d_{\max}(a,b_1/C)\leq  d_{\max}(b_1,b_1/C)\p d(a,b_1)$. It follows that $d_{\max}(a,b_1/C)\preceq d_{\max}(b_1,b_1/C)$, and so $d(a,b_1)\prec d_{\max}(b_1,b_1/C)$. But then $d(a,b_1)\p d(a,b_1)<d_{\max}(b_1,b_1/C)$, which contradicts $A\ind^{d_{\max}}_C B$. 
\end{proof} 

We can now give the characterization of simplicity for $\MU{\cR}$. The reader should compare the statement of this result to the statement of Theorem \ref{thm:stable}.

\begin{theorem}\label{thm:simple}
Given a Urysohn monoid $\cR$, the following are equivalent.
\begin{enumerate}[$(i)$]
\item $\TR{\cR}$ is simple.
\item $\ind^{\dist}$ implies $\ind^f$.
\item $\ind^f$ coincides with $\ind^{d_{\max}}$.
\item For all $r,s\in R$ if $r\leq s$ then $r\p r\p s=r\p s$.
\end{enumerate}
\end{theorem}
\begin{proof}
$(i)\Rightarrow(iv)$: Suppose $(iv)$ fails, and fix $r,s\in R$ such that $r\leq s$ and $r\p s<2r\p s$. Choose $a,b_1,b_2,c\in\U_\cR$ such that
\begin{equation*}
\begin{aligned}
&d(a,b_1) =d(a,c)=r,\\
&d(a,b_2) =d(b_2,c)=s,
\end{aligned}
\hspace{1cm}
\begin{aligned}
&d(b_1,c) =2r,\\
&d(b_1,b_2) =r\p s.
\end{aligned}
\end{equation*}
Then $d(a,b_1)\p d(a,b_2)=r\p s<2r\p s=d(b_1,c)\p d(b_2,c)$, and so $a\nind^f_c b_1b_2$ by Theorem \ref{thm:forks}. On the other hand $d_{\max}(a,a/b_1b_2c)=2r=d(a,c)\p d(a,c)$, and so $b_1b_2\ind^f_c a$. Therefore $\TR{\cR}$ is not simple by Fact \ref{fact:low}$(a)$.

$(iv)\Rightarrow(ii)$: Assume $(iv)$ and suppose $A\ind^{\dist}_C B$. By Proposition \ref{simpleFork}, it suffices to fix $a\in A$ and $b_1,b_2\in B$, and show $d_{\max}(b_1,b_2/C)\leq d(a,b_1)\p d(a,b_2)$.

Without loss of generality, assume $d(a,b_1)\leq d(a,b_2)$. Since $A\ind^{\dist}_C B$, we have $d(a,C)\leq d(a,b_1)$, which means $d_{\max}(a,a/C)\leq 2d(a,b_1)$. As in the proof of Proposition \ref{simpleFork}, if $\alpha,\beta\in R^*$ then $\alpha\leq\beta$ implies $\alpha\p\alpha\p\beta=\alpha\p\beta$. Combining these observations with Lemma \ref{lem:dmax}, we have
\begin{align*}
d_{\max}(b_1,b_2/C) &\leq d_{\max}(a,b_1/C)\p d(a,b_2)\\
 &\leq d_{\max}(a,a/C)\p d(a,b_1)\p d(a,b_2)\\
 &\leq 3d(a,b_1)\p d(a,b_2)\\
  &= d(a,b_1)\p d(a,b_2).
 \end{align*}

$(ii)\Rightarrow(i)$: If $(ii)$ holds then, by Proposition \ref{prop:SIR}$(b)(iv)$, $\ind^f$ satisfies local character. Therefore $\TR{\cR}$ is simple by Fact \ref{fact:low}$(a)$.

$(iii)\Leftrightarrow(iv)$: By Proposition \ref{simpleFork}.
\end{proof}

As a corollary, we obtain another characterization of $\ind^f$ for simple $\TR{\cR}$. In particular, this result will be the key tool in the analysis of supersimplicity for $\TR{\cR}$. See also Section \ref{rem:thorn}. 

\begin{corollary}\label{cor:simple}
Given a Urysohn monoid $\cR$, the following are equivalent.
\begin{enumerate}[$(i)$]
\item $\TR{\cR}$ is simple.
\item For all $A,B,C\subset\U_\cR$, $A\ind^f_C B$ if and only if $2d(a,BC)=2d(a,C)$ for all $a\in A$.
\end{enumerate}
\end{corollary}
\begin{proof}
If $(ii)$ holds then $\ind^{\dist}$ implies $\ind^f$, and so $\TR{\cR}$ is simple by Theorem \ref{thm:simple}. Conversely, suppose $\TR{\cR}$ is simple. Then $\ind^f$ coincides with $\ind^{d_{\max}}$ by Theorem \ref{thm:simple}, and so $\ind^{d_{\max}}$ is symmetric. Fix $A,B,C\subset\U_\cR$. Then
\begin{align*}
\textstyle A\ind^{d_{\max}}_C B &\miff \textstyle a\ind^{d_{\max}}_C B\text{ for all $a\in A$}\\
&\miff \textstyle B\ind^{d_{\max}}_C a\text{ for all $a\in A$}\\
&\miff 2d(a,BC)=2d(a,C)\text{ for all $a\in A$}.\qedhere
\end{align*}
\end{proof}

The final result of this section is motivated by the distance monoid $\cR_n$ in the case when $n=1,2$ (see Example \ref{allEX}(3)). Recall that $\cU_{\cR_2}$ can be viewed as the countable random graph. Moreover, $\cU_{\cR_1}$ is simply a countably infinite complete graph, and therefore its theory is interdefinable with the theory of infinite sets in the empty language. With these interpretations, $\TR{\cR_1}$ and $\TR{\cR_2}$ are both classical examples in which nonforking is as uncomplicated as possible. In particular, $A\ind^f_C B$ if and only if $A\cap B\seq C$ (see e.g. \cite{Wabook}). We generalize this behavior as follows.  

\begin{definition}
A distance monoid $\cR$ is \textbf{metrically trivial} if $r\p s=\sup R$ for all $r,s\in R^{>0}$ (in particular, $R$ contains a maximal element).
\end{definition} 

The following properties of metrically trivial monoids are easy to verify.

\begin{proposition}
Let $\cR$ be a countable distance monoid.
\begin{enumerate}[$(a)$]
\item $\cR$ is metrically trivial if and only if $r\leq s\p t$ for all $r,s,t\in R^{>0}$.
\item If $\cR$ is metrically trivial then $\cR^*$ is metrically trivial (use Theorem \ref{thm:R*}(b)).
\item If $\cR$ is metrically trivial then $\cR$ is a Urysohn monoid (use Theorem \ref{thm:QE}).
\end{enumerate}
\end{proposition}

In particular, property $(a)$ says $\cR$ is metrically trivial if and only if $\cR$-metric spaces coincide with graphs whose edges are arbitrarily colored by nonzero elements of $R$. So $\cU_\cR$ is the random edge-colored graph, with color set $R^{>0}$. 

\begin{theorem}
Given a Urysohn monoid $\cR$, the following are equivalent.
\begin{enumerate}[$(i)$]
\item $\cR$ is metrically trivial.
\item For all $A,B,C\subset\MU{\cR}$, $A\ind^f_C B$ if and only if $A\cap B\seq C$.
\end{enumerate}
\end{theorem}
\begin{proof}
$(i)\Rightarrow(ii)$: Suppose $\cR$ is metrically trivial, and fix $A,B,C\subset\U_\cR$ such that $A\nind^f_C B$. Note that metrically trivial monoids satisfy condition $(iv)$ of Theorem \ref{thm:simple}. By Corollary \ref{cor:simple}, there is $a\in A$ and $b\in B\backslash C$ such that $2d(a,b)<2d(a,C)$. Since $\cR^*$ is metrically trivial, we must have $d(a,b)=0$, and so $a=b$.

$(ii)\Rightarrow(i)$: Suppose $\cR$ is not metrically trivial. Then there is $r\in R^{>0}$ such that $r\p r<\sup R$. Fix $a,b\in\U_\cR$ such that $d(a,b)=r$. Then $\{a\}\cap\{b\}=\emptyset$. On the other hand, $d(a,b)\p d(a,b)<\sup R=d_{\max}(b,b/\emptyset)$, and so $a\nind^f_{\emptyset}b$.
\end{proof}

\subsection{Non-axiomatizable Properties}\label{sec:super}

We have now shown that stability and simplicity (and thus all of their equivalent formulations) are each finitely axiomatizable as properties of $\mathbf{RUS}$. In this section, we show that supersimplicity and superstability are also characterized by properties of $\cR$, which, while very natural, are not axiomatizable. 

The idea behind the characterization of supersimplicity is the observation that, due to Corollary \ref{cor:simple}, forking chains in simple Urysohn spaces are witnessed by decreasing distances in $\cR^*$ of the form $2\alpha$, where $\alpha\in\cR^*$. Moreover, if $\TR{\cR}$ is simple then, by Theorem \ref{thm:simple}$(iv)$ and upper semicontinuity, we have $2\alpha\p 2\alpha=2\alpha$ for any $\alpha\in \cR^*$. Altogether, forking chains in simple Urysohn spaces are witnessed by decreasing sequences of idempotent elements of $\cR^*$.

\begin{definition}
Let $\eq(\cR)=\{r\in R:2r=r\}$ be the submonoid of idempotent elements of a distance monoid $\cR$. Let $\eq^<(\cR)=\{r\in\eq(\cR):r\neq\sup R\}$. 
\end{definition}

Call a linear order $I=(I,<)$ \emph{successive} if every non-maximal element $i\in I$ has an immediate successor, denoted $i+1$. If $I$ is successive, $\cR$ is a distance monoid, and $\alpha\in R^*$, then we say \emph{$I$ embeds in $\eq(\cR^*)$ below $\alpha$} if there is an order-preserving sequence $(\beta_i)_{i\in I}$ in $\eq(\cR^*)$ with $\beta_i<\alpha$ for all $i\in I$. If $\beta_i\in\eq(\cR)$ for all $i\in I$, then we say \emph{$I$ embeds in $\eq(\cR)$ below $\alpha$}. 

\begin{lemma}\label{lem:supsim}
Suppose $\cR$ is a countable distance monoid such that, for some $n>0$, $nr\in\eq(\cR)$ for all $r\in R$. If $\alpha\in R^*$ and  $I$ is a successive linear order, which embeds in $\eq(\cR^*)$ below $\alpha$, then $I$ embeds in $\eq(\cR)$ below $\alpha$.
\end{lemma} 
\begin{proof}
Let $(\beta_i)_{i\in I}$ be an order-preserving sequence in $\eq(\cR^*)$, with $\beta_i<\alpha$ for all $i\in I$. Given $i\in I$, we have $\beta_i=n\beta_i=\inf\{ns:s\in R,~\beta_i\leq s\}$ by upper semicontinuity in $\cR^*$. Therefore, there is $s\in R$ such that $\beta_i\leq ns<\beta_{i+1}$ (where $\beta_{i+1}=\alpha$ if $i=\max I$). Let $r_i=ns$. Altogether, the sequence $(r_i)_{i\in I}$ witnesses that $I$ embeds in $\eq(\cR)$ below $\alpha$.
\end{proof}

\begin{theorem}\label{thm:supsim}
Suppose $\cR$ is a Uryoshn monoid and $\TR{\cR}$ is simple. Then $\TR{\cR}$ is supersimple if and only if $\eq(\cR)$ is well-ordered. Moreover, in this case $SU(\TR{\cR})$ is the order type of $\eq^<(\cR)$.
\end{theorem}
\begin{proof}
Suppose $\Th(\cU_\cR)$ is not supersimple. Then there is an increasing chain $(B_i)_{i<\omega}$ of subsets of $\U_\cR$, and some $a\in\U_\cR$, such that $a\nind^f_{B_i}B_j$ for all $i<j$. Let $\alpha_i=2d(a,B_i)$. By Corollary \ref{cor:simple}, we have $\alpha_i<\alpha_j$ for all $i<j$, and so $(\alpha_i)_{i<\omega}$ forms a strictly decreasing sequence in $\eq(\cR^*)$. Since $\TR{\cR}$ is simple, we have the hypotheses of Lemma \ref{lem:supsim} with $n=2$. Therefore $\eq(\cR)$ is not well-ordered.

Conversely, suppose $(r_i)_{i<\omega}$ is a strictly decreasing sequence in $\eq(\cR)$. Then we may find $a\in\U_\cR$ and $b_i\in\U_\cR$, for $i<\omega$, such that $d(a,b_i)=r_i$. Given $i<\omega$, set $B_i=\{b_j:i\leq j\}$. Then, by Corollary \ref{cor:simple}, we have $a\nind^f_{B_i}B_j$ for all $i<j$. It follows that $\Th(\cU_\cR)$ is not supersimple.

Finally, suppose $\TR{\cR}$ is supersimple. Recall that there is a unique $1$-type in $S_1^{\cU_\cR}(\emptyset)$ and so, if we fix $a\in\U_\cR$, then $SU(\TR{\cR})=SU(a/\emptyset)$. To show $SU(a/\emptyset)$ is the order type of $\eq^<(\cR)$, it suffices by Lemma \ref{lem:supsim} to show $SU(a/\emptyset)\geq\mu$ if and only if $\mu$ embeds in $\eq(\cR^*)$ below $\sup R^*$. To accomplish this, we show by induction on $\mu$, that for any $B\subset\U_\cR$, $SU(a/B)\geq\mu$ if and only if $\mu$ embeds in $\eq(\cR^*)$ below $2d(a,B)$. 

The case when $\mu$ is a limit ordinal follows easily by induction and the fact that $\eq(\cR^*)$ is well-ordered. So suppose $\mu=\eta+1$. If $SU(a/B)\geq\mu$ then there is some $B_1\supset B$ such that $a\nind^f_B B_1$ and $SU(a/B_1)\geq\eta$. By induction, $\eta$ embeds in $\eq(\cR^*)$ below $2d(a,B_1)$. By Corollary \ref{cor:simple}, we have $2d(a,B_1)<2d(a,B)$, and so $\mu=\eta+1$ embeds in $\eq(\cR^*)$ below $2d(a,B)$.  Conversely, if $\mu$ embeds in $\eq(\cR^*)$ below $2d(a,B)$, then there is some $\alpha\in\eq(\cR^*)$ such that $\alpha<2d(a,B)$ and $\eta$ embeds in $\eq(\cR^*)$ below $\alpha$. Fix $b\in\U_\cR$ such that $d(a,b)=\alpha$ and set $B_1=Bb$. Then $\eta$ embeds in $\eq(\cR^*)$ below $2d(a,B_1)$, and so $SU(a/B_1)\geq\eta$ by induction. We also have $a\nind^f_B B_1$ by Corollary \ref{cor:simple}, and so $SU(a/B)\geq\eta+1=\mu$.
\end{proof}

Note, in particular, that if $\TR{\cR}$ is supersimple of $SU$-rank $1$ then it must be the case that $2r=\sup R$ for all $r\in R^{>0}$. Altogether, $\TR{\cR}$ is supersimple of $SU$-rank $1$ if and only if $\cR$ is metrically trivial.

Recall that, by Theorem \ref{thm:stable}, $\TR{\cR}$ is stable if and only if $\eq(\cR)=\cR$. This yields the following characterization of superstability.

\begin{theorem}\label{thm:supstab}
Suppose $\cR$ is a Urysohn monoid and $\TR{\cR}$ is stable. The following are equivalent.
\begin{enumerate}[$(i)$]
\item $\TR{\cR}$ is $\omega$-stable.
\item $\TR{\cR}$ is superstable.
\item $\cR$ is well-ordered.
\end{enumerate}
\end{theorem}
\begin{proof}
$(ii)\Leftrightarrow(iii)$: By Theorems \ref{thm:stable}$(iv)$ and \ref{thm:supsim}.

$(i)\Rightarrow(ii)$: This is true for any theory.

$(iii)\Rightarrow(i)$: Suppose $\cR$ is well-ordered. Consider $\TR{\cR}$ as the theory of refining equivalence relations $d(x,y)\leq r$, indexed by $(R,\leq,0)$ (also commonly known as ``expanding equivalence relations"). This example is well-known in the folklore to be $\omega$-stable. The case $R=(\N,\leq,0)$ is credited to Shelah (see e.g. \cite{HyLe}). A general argument can be found in \cite[Theorem 3.5.16]{Cothesis}.
\end{proof}

\begin{corollary}\label{cor:undef}
Supersimplicity and superstability are not axiomatizable properties of $\mathbf{RUS}$.
\end{corollary}
\begin{proof}
Since stability is finitely axiomatizable, it is enough to show that superstability is not axiomatizable. Suppose, toward a contradiction, that there is an $(\LDS)_{\omega_1,\omega}$-sentence $\vphi$ such that, for any Urysohn monoid $\cR$, $\TR{\cR}$ is superstable if and only if $\cR\models\vphi$. After adding constants $(c_i)_{i<\omega}$ to $\LDS$, and conjuncting with $\oms{QE}$ along with a sentence axiomatizing distance monoids with universe $(c_i)_{i<\omega}$, we obtain an $(\LDS)_{\omega_1,\omega}$-sentence $\vphi^*$ such that, for any $\LDS$-structure $\cR$, $\cR\models\vphi^*$ if and only if $\cR$ is a countable, ultrametric, well-ordered, distance monoid. By classical results in infinitary logic (see e.g. \cite[Corollary 4.28]{Mainf}), it follows that there is some $\mu<\omega_1$ such that any model of $\vphi^*$ has order type at most $\mu$. This is clearly a contradiction, since any ordinal can be given the structure of an ultrametric distance monoid (c.f. Example \ref{allEX}(5)).
\end{proof}

On the other hand, for any $n<\omega$, ``$\TR{\cR}$ is supersimple of $SU$-rank $n$" is finitely axiomatizable in $\RUS$.

\section{Cyclic Indiscernible Sequences}\label{sec:NFSOP}

So far our results have been motivated by choosing a particular kind of good behavior for $\TR{\cR}$ and then characterizing when this behavior happens. In this section, we give a uniform upper bound for the complexity of $\TR{\cR}$ for any Urysohn monoid $\cR$. In particular, we will show that if $\cR$ is a Urysohn monoid then $\TR{\cR}$ does not have the finitary strong order property. We will accomplish this by proving the following theorem.

\begin{theorem}\label{thm:cycbound}
Suppose $\cR$ is a Urysohn monoid, and $\cI=(\abar^l)_{l<\omega}$ is an indiscernible sequence in $\MU{\cR}$ of tuples of possibly infinite length. If $|\NP(\cI)|=n<\omega$ then $\cI$ is $(n+1)$-cyclic.
\end{theorem}

By definition of $\FSOP$, we will then obtain the following corollary.

\begin{corollary}\label{cor:NFSOP}
If $\cR$ is a Urysohn monoid then $\TR{\cR}$ is $\NFSOP$.
\end{corollary}

The proof of Theorem \ref{thm:cycbound} is obtained by translating \cite[Section 3.1]{CoTe} (joint work with Caroline Terry), in which it is shown that the complete Urysohn sphere in continuous logic does not have  the \emph{fully} finitary strong order property. This property is stronger than the finitary strong order property, and so, after translating back to the setting of \cite{CoTe}, the work in this section also slightly strengthens the results of \cite{CoTe}. A translation guide is given in \ref{app:FD}.

For the rest of the section, we fix a Urysohn monoid $\cR$. The key tool we use to prove Theorem \ref{thm:cycbound} is the following test for when an indiscernible sequence in $\MU{\cR}$ is $n$-cyclic. 

\begin{lemma}\label{lem:cycchar}
Let $\cR$ be a Urysohn monoid, and $(\abar^l)_{l<\omega}$ an indiscernible sequence in $\MU{\cR}$. Given $i,j\in\ell(\abar^0)$, set $\epsilon_{i,j}=d(a^0_i,a^1_j)$. Given $n\geq 2$, $(\abar^l)_{l<\omega}$ is $n$-cyclic if and only if, for all $i_1,\ldots,i_n\in\ell(\abar^0)$, $\epsilon_{i_n,i_1}\leq\epsilon_{i_1,i_2}\p\epsilon_{i_2,i_3}\p\ldots\p\epsilon_{i_{n-1},i_n}$.
\end{lemma}

The proof of this result is a direct translation of \cite[Lemma 3.7]{CoTe} to the setting of discrete logic and generalized metric spaces. The only detail requiring special attention is the following. In the proof of  \cite[Lemma 3.7]{CoTe}, we use that fact that, given a partial symmetric function $f:\dom(f)\seq X\times X\func[0,1]$ (where $X$ is some set), $f$ can be extended to a pseudometric on $X$ if and only if, for all $m>0$ and for all $z_0,z_1,\ldots,z_m\in X$, if $(z_0,z_m)\in\dom(f)$ and $(z_i,z_{i+1})\in\dom(f)$ for all $0\leq i<m$, then
$$
f(z_0,z_m)\leq f[\zbar]:=f(z_0,z_1)+ f(z_1,z_2)+\ldots+ f(z_{m-1},z_m).
$$
To see this, note that the forward direction follows from the triangle inequality and, for the reverse direction, define the pseudometric $d_f:X\times X\func [0,1]$ such that $d_f(x,y)$ is the infimum over all values $f[\zbar]$, where $\zbar=(z_0,\ldots,z_m)$ is such that $z_0=x$, $z_m=y$, and  $(z_i,z_{i+1})\in\dom(f)$ for all $0\leq i<m$.

The analog of this fact can be used to characterize when a  partial symmetric function $f:\dom(f)\seq X\times X\func R^*$ can be extended to an $\cR^*$-pseudometric on $X$. The only subtlety is that upper semicontinuity in $\cR^*$ is necessary to prove that the function $d_f$ described above satisfies the triangle inequality.

It is also worth noting that \cite[Lemma 3.7]{CoTe} is only stated for indiscernible sequences of tuples of finite length. But this assumption is not used in the proof.

The final tools needed for Theorem \ref{thm:cycbound} are the following observations concerning transitivity properties of indiscernible sequences.   

\begin{lemma}\label{lem:cycind}
Suppose $\cI=(\abar^l)_{l<\omega}$ is an indiscernible sequence in $\MU{\cR}$. Given $i,j\in\ell(\abar^0)$, set $\epsilon_{i,j}=d(a^0_i,a^1_j)$. Fix $n\geq 2$ and $i_1,\ldots,i_n\in\ell(\abar^0)$.
\begin{enumerate}[$(a)$]
\item $\epsilon_{i_1,i_n}\leq\epsilon_{i_1,i_2}\p\epsilon_{i_2,i_3}\p\ldots\p\epsilon_{i_{n-1},i_n}$.
\item If $i_s=i_t$ for some $1\leq s<t\leq n$, then $\epsilon_{i_n,i_1}\leq\epsilon_{i_1,i_2}\p\epsilon_{i_2,i_3}\p\ldots\p\epsilon_{i_{n-1},i_n}$.
\item If $i_s\not\in\NP(\cI)$ for some $1\leq s\leq n$, then $\epsilon_{i_n,i_1}\leq\epsilon_{i_1,i_2}\p\epsilon_{i_2,i_3}\p\ldots\p\epsilon_{i_{n-1},i_n}$.
\end{enumerate}
\end{lemma}
\begin{proof}
For parts $(a)$ and $(b)$, see \cite[Lemma 3.5]{CoTe}.

Part $(c)$. First, if $i_s\not\in\NP(\cI)$ then $a^0_{i_s}=a^2_{i_s}$. Therefore, for any $j\in\ell(\abar^0)$, we have
$\epsilon_{i_s,j}=d(a^0_{i_s},a^1_j)=d(a^2_{i_s},a^1_j)=d(a^0_j,a^1_{i_s})=\epsilon_{j,i_s}$. So if $s=1$ or $s=n$ then the result follows immediately from part $(a)$. Suppose $1<s<n$. Then, using part $(a)$, we have $
\epsilon_{i_n,i_1} \leq \epsilon_{i_n,i_s}\p\epsilon_{i_s,i_1} = \epsilon_{i_1,i_s}\p\epsilon_{i_s,i_n} = \epsilon_{i_1,i_2}\p\ldots\p\epsilon_{i_{n-1},i_n}$. 
\end{proof}

We can now prove the main result of this section.

\begin{proof}[Proof of Theorem \ref{thm:cycbound}]
Let $\cR$ be a Urysohn monoid and fix an indiscernible sequence $\cI$ in $\MU{\cR}$, with $|\NP(\cI)|=n<\omega$. We want to show $\cI$ is $(n+1)$-cyclic. We may assume $n\geq 1$ and so, by Lemma \ref{lem:cycchar}, it suffices to fix $i_1,\ldots,i_{n+1}\in\ell(\abar^0)$ and show $\epsilon_{i_{n+1},i_1}\leq\epsilon_{i_1,i_2}\p\epsilon_{i_2,i_3}\p\ldots\p\epsilon_{i_n,i_{n+1}}$. By Lemma \ref{lem:cycind}$(c)$, we may assume $i_s\in\NP(\cI)$ for all $1\leq s\leq n+1$. Therefore, there are $1\leq s<t\leq n+1$ such that $i_s=i_t$, and so the result follows from Lemma \ref{lem:cycind}$(b)$. 
\end{proof}

\section{Strong Order Rank}\label{sec:SOR}

Suppose $\cR$ is a Urysohn monoid. Summarizing previous results, we have shown that $\TR{\cR}$ never has the finitary strong order property and, moreover, stability and simplicity are both possible. In this section, we address the complexity between simplicity and $\NFSOP$. For general first-order theories, this complexity is stratified by Shelah's $\SOP_n$-hierarchy, which we have formulated as \emph{strong order rank} (Definition \ref{def:SOrank}).

\subsection{Calculating the Rank} 

The results of Sections \ref{sec:stable} and \ref{sec:simple} can be restated as follows:
\begin{enumerate}[$(i)$]
\item $\TR{\cR}$ is stable if and only if for all $r,s\in R$, if $r\leq s$ then $r\p s=s$.
\item $\TR{\cR}$ is simple if and only if for all $r,s,t\in R$, if $r\leq s\leq t$ then $r\p s \p t=s\p t$.
\end{enumerate}

This motivates the following definition.

\begin{definition}\label{def:arch}
The \textbf{archimedean complexity} of a distance monoid $\cR$, denoted $\arch(\cR)$, is the minimum $n<\omega$ such that $r_0\p r_1\p\ldots\p r_n=r_1\p\ldots\p r_n$ for all $r_0\leq r_1\leq \ldots\leq r_n$ in $R$. If no such $n$ exists, set $\arch(\cR)=\omega$. 
\end{definition}

Roughly speaking, $\arch(\cR)$ measures when, if ever, addition in $\cR$ begins to stabilize. For example, if $\arch(\cR)\leq n$ then $nr\in\eq(\cR)$ for all $r\in R$ (and so $\cR$ satisfies the hypotheses of Lemma \ref{lem:supsim}). In the case that $\cR$ is archimedean, this can be stated in a more precise fashion.

\begin{definition}
A distance monoid $\cR$ is \textbf{archimedean} if, for all $r,s\in R^{>0}$, there is some $n>0$ such that $s\leq nr$ (i.e. $s\preceq r$).
\end{definition}

If $\cR$ is an archimedean distance monoid, then $\arch(\cR)$ provides a uniform upper bound for the number of times any given positive element of $\cR$ must be added to itself in order to surpass any other element of $\cR$. In other words, $\arch(\cR)\leq n$ if and only if $s\leq nr$ for all $r,s\in R^{>0}$ (i.e. $nr=\sup R$ for any $r\in R^{>0}$). 

We have shown that, for general Urysohn monoids $\cR$, $\TR{\cR}$ is stable if and only if $\arch(\cR)\leq 1$, and so, by Fact \ref{fact:SOP}, $\SO(\cU_\cR)\leq 1$ if and only if $\arch(\cR)\leq 1$. The goal of this section is to extend this result, and show $\SO(\cU_\cR)=\arch(\cR)$ for any Urysohn monoid $\cR$. First, we note the following observation, which says that there is no loss in considering the archimedean complexity of $\cR^*$ rather than that of $\cR$. The proof, which uses the continuity properties of Theorem \ref{thm:R*}$(b)$, is left to the reader. 

\begin{proposition}\label{prop:same}
If $\cR$ is a countable distance monoid then $\arch(\cR^*)=\arch(\cR)$.
\end{proposition}

For the rest of the section, we fix a Urysohn monoid $\cR$. Toward the proof of the main result of this section (Theorem \ref{thm:SOR}), we begin by refining previous results on cyclic indiscernible sequences.

\begin{definition}
Fix $n\geq 2$ and $\alpha_1,\ldots,\alpha_n\in R^*$. Let $\albar=(\alpha_1,\ldots,\alpha_n)$.
\begin{enumerate}
\item $\albar$ is a \textbf{diagonally indiscernible tuple} if there is an indiscernible sequence $(\abar^l)_{l<\omega}$ in $\MU{\cR}$, with $\ell(\abar^0)=n$, such that $d(a^1_1,a^0_n)=\alpha_n$ and, for all $1\leq t<n$, $d(a^0_t,a^1_{t+1})=\alpha_t$.

\item $\albar$ is \textbf{transitive} if $\alpha_n\leq\alpha_1\p\ldots\p\alpha_{n-1}$.
\end{enumerate}
\end{definition}

\begin{proposition}\label{prop:SOcyc}
Given $n\geq 2$, $\SO(\cU_\cR)\leq n-1$ if and only if every diagonally indiscernible tuple in $(R^*)^n$ is transitive.
\end{proposition}
\begin{proof}
Assume $\SO(\cU_\cR)\leq n-1$, and fix a diagonally indiscernible tuple $\albar=(\alpha_1,\ldots,\alpha_n)\in(R^*)^n$, witnessed by an indiscernible sequence $(\abar^l)_{l<\omega}$ in $\MU{\cR}$. Then $(\abar^l)_{l<\omega}$ is $n$-cyclic by definition of strong order rank, and so there is $(\cbar^1,\ldots,\cbar^n)$ such that $(\cbar^t,\cbar^{t+1})\equiv(\abar^0,\abar^1)\equiv(\cbar^n,\cbar^1)$ for all $t<n$. In particular,
$$
\alpha_n=d(a^0_n,a^1_1)=d(c^1_1,c^n_n)\leq d(c^1_1,c^2_2)\p\ldots\p d(c^{n-1}_{n-1},c^n_n)=\alpha_1\p\ldots\p \alpha_{n-1}.
$$
Therefore $\albar$ is transitive.

Conversely, suppose $\SO(\cU_\cR)\geq n$. By definition, there is an indiscernible sequence $\cI=(\abar^l)_{l<\omega}$ in $\MU{\cR}$, which is not $n$-cyclic. Given $i,j\in\ell(\abar^0)$, let $\epsilon_{i,j}=d(a^0_i,a^1_j)$. By Lemma \ref{lem:cycchar}, there are $i_1,\ldots,i_n\in\ell(\abar^0)$ such that
$$
\epsilon_{i_n,i_1}>\epsilon_{i_1,i_2}\p\ldots\p\epsilon_{i_{n-1},i_n}.
$$
By Lemma \ref{lem:cycind}$(b)$, the map $t\mapsto i_t$ is injective. Given $l<\omega$, define $\bbar^l=(a^l_{i_1},\ldots,a^l_{i_n})$. Then $\ell(\bbar^0)=n$ and $\cJ=(\bbar^l)_{l<\omega}$ is an indiscernible sequence. Let $\alpha_n=\epsilon_{i_n,i_1}$ and, given $1\leq t<n$, let $\alpha_t=\epsilon_{i_t,i_{t+1}}$. Then, for any $t<n$, we have $d(b^0_t,b^1_{t+1})=d(a^0_{i_t},a^1_{i_{t+1}})=\alpha_t$. Moreover, $d(b^0_n,b^1_1)=d(a^0_{i_n},a^1_{i_1})=\alpha_n$. Therefore $\cJ$ witnesses that $\albar=(\alpha_1,\ldots,\alpha_n)$ is a non-transitive diagonally indiscernible tuple in $(R^*)^n$. 
\end{proof}

Next, we prove two technical lemmas. 

\begin{lemma}\label{lem:diag}
Suppose $n\geq 2$ and $(\alpha_1,\ldots,\alpha_n)$ is a diagonally indiscernible tuple in $(R^*)^n$. Then, for any $1\leq i<n$, we have
$\alpha_n\leq \alpha_1\p\ldots\p \alpha_{n-1}\p 2\alpha_i$.
\end{lemma}
\begin{proof}
Let $(\abar^l)_{l<\omega}$ be an indiscernible sequence in $\MU{\cR}$, which witnesses that $(\alpha_1,\ldots,\alpha_n)$ is a diagonally indiscernible tuple. Given $1\leq i,j\leq n$, let $\epsilon_{i,j}=d(a^0_i,a^1_j)$. Note that, if $1\leq i<n$ then $\epsilon_{i,i+1}=\alpha_i$ and, moreover,
$$
\epsilon_{i+1,i+1}=d(a^1_{i+1},a^2_{i+1})\leq d(a^1_{i+1},a^0_i)\p d(a^0_i,a^2_{i+1})=2\alpha_i.
$$
If $i<n-1$ then, using Lemma \ref{lem:cycind}$(a)$, we have
\begin{align*}
\alpha_n &= d(a^2_1,a^1_n)\\
 &\leq d(a^2_1,a^3_{i+1})\p d(a^3_{i+1},a^0_{i+1})\p d(a^0_{i+1},a^1_n)\\
 &= \epsilon_{1,i+1}\p \epsilon_{i+1,n}\p \epsilon_{i+1,i+1}\\
 &\leq \alpha_1\p\ldots\p\alpha_{n-1}\p 2\alpha_i
 \end{align*}
On the other hand, if $i=n-1$ then, using Lemma \ref{lem:cycind}$(a)$, we have
\begin{align*}
\alpha_n &= d(a^1_1,a^0_n)\\
 &\leq d(a^1_1,a^2_n)\p d(a^2_n,a^0_n)\\
  & = \epsilon_{1,n}\p\epsilon_{n,n} \leq \alpha_1\p\ldots\p\alpha_{n-1}\p 2\alpha_{n-1}.\qedhere
 \end{align*}
\end{proof}

\begin{lemma}\label{lem:SOseq}
Fix $n\geq 2$. If $\alpha_1,\ldots,\alpha_n\in R^*$, with $\alpha_1\leq\alpha_2\leq\ldots\leq \alpha_n$, and let $\beta=\alpha_1\p\ldots\p\alpha_n$, then  $(\alpha_2,\ldots,\alpha_n,\beta)$ is a diagonally indiscernible tuple.
\end{lemma}
\begin{proof}
Define the sequence $(\abar^l)_{l<\omega}$, such that $\ell(\abar^0)=n$ and, given $k\leq l<\omega$ and $1\leq i,j\leq n$,
$$
d(a^k_i,a^l_j)=\begin{cases}
\alpha_j\p\alpha_{j+1}\p\ldots\p\alpha_i & \text{if $k<l$ and $i\geq j$, or $k=l$ and $i>j$}\\
\alpha_{i+1}\p\alpha_{i+2}\p\ldots\p\alpha_j & \text{if $k<l$ and $i<j$.}
\end{cases}
$$
We have $d(a^0_n,a^1_1)=\beta$ and, given $1\leq i<n$, $d(a^0_i,a^1_{i+1})=\alpha_{i+1}$. Therefore, it suffices to verify this sequence satisfies the triangle inequality. This verification is quite tedious, but follows from a routine case analysis, which crucially depends on the assumption $
\alpha_1\leq \alpha_2\leq\ldots\leq \alpha_n$.
\end{proof}

We now have the tools necessary to prove the main result of this section.

\begin{theorem}\label{thm:SOR}
If $\cR$ is a Urysohn monoid then $\SO(\cU_\cR)=\arch(\cR)$.
\end{theorem}
\begin{proof}
First, note that if $\SO(\cU_\cR)>n$ and $\arch(\cR)>n$ for all $n<\omega$ then, by Corollary \ref{cor:NFSOP} and our conventions, we have $\SO(\cU_\cR)=\omega=\arch(\cR)$. 

Therefore, it suffices to fix $n\geq 1$ and show $\SO(\cU_\cR)\geq n$ if and only if $\arch(\cR)\geq n$. Note that $\arch(\cR)=0$ if and only if $\cR$ is the trivial monoid, in which case $\MU{\cR}$ is a single point. Conversely, if $\cR$ is nontrivial then $\MU{\cR}$ is clearly infinite. So $\arch(\cR)<1$ if and only if $\TR{\cR}$ has finite models, which is equivalent to $\SO(\cU_\cR)<1$. So we may assume $n\geq 2$. 

Suppose $\arch(\cR)\geq n$. Then there are $r_1,\ldots,r_n\in R$ such that $r_1\leq r_2\leq\ldots\leq r_n$ and $r_2\p\ldots\p r_n<r_1\p r_2\p\ldots\p r_n$. By Lemma \ref{lem:SOseq}, it follows that $(r_2,\ldots,r_n,r_1\p\ldots\p r_n)$ is a non-transitive diagonally indiscernible tuple of length $n$. Therefore $\SO(\cU_\cR)\geq n$ by Proposition  \ref{prop:SOcyc}.

Finally, suppose $\SO(\cU_\cR)\geq n$. By Proposition \ref{prop:SOcyc} there is a non-transitive diagonally indiscernible tuple $(\alpha_1,\ldots,\alpha_n)$ in $(R^*)^n$. Let $\beta_1,\ldots,\beta_{n-1}$ be an enumeration of $\alpha_1,\ldots,\alpha_{n-1}$, with $\beta_1\leq\ldots\leq \beta_{n-1}$. Then, using Lemma \ref{lem:diag},
$$
\beta_1\p\ldots\p\beta_{n-1}<\alpha_n\leq 2\beta_1\p\beta_1\p\ldots\p\beta_{n-1},
$$
which implies $\beta_1\p\ldots\p\beta_{n-1}<\beta_1\p\beta_1\p\ldots\p\beta_{n-1}$, and so $\arch(\cR^*)\geq n$. By Proposition \ref{prop:same}, $\arch(\cR)\geq n$. 
\end{proof}

Since archimedean complexity is a first-order property of distance monoids, ``$\SO(\cU_\cR)=n$" is a finitely axiomatizable property of $\mathbf{RUS}$ for all $n<\omega$. It then follows that ``$\SO(\cU_\cR)=\omega$" is an axiomatizable property of $\mathbf{RUS}$.

\subsection{Further Remarks on Simplicity}

If $\cR$ is a Urysohn monoid then, by Theorem \ref{thm:simple}, $\TR{\cR}$ simple if and only if $\arch(\cR)\leq 2$. Therefore, by Theorem \ref{thm:SOR}, we have the following corollary.  

\begin{corollary}\label{cor:SOP3}
If $\cR$ is a Urysohn monoid, and $\TR{\cR}$ is $\NSOP_3$, then $\TR{\cR}$ is simple.
\end{corollary}

This corollary deserves mention because, in general, non-simple $\NSOP_3$ theories are quite scarce. A similar phenomenon in model theoretic dividing lines concerns non-simple theories, which have neither $\TP_2$ nor the strict order property. In particular, there are no known examples of such theories. Since we have shown that $\TR{\cR}$ never has the strict order property, it is worth proving that $\NTP_2$ for $\TR{\cR}$ implies simplicity. This also fits nicely with the fact that $\NIP$ for $\TR{\cR}$ implies stability (see Remark \ref{rem:NIP}). We refer the reader to \cite{ChNTP2} for the definition of $\TP_2$, and recall the fact that simple theories are $\NTP_2$.

\begin{theorem}\label{thm:TP2}
If $\cR$ is a Urysohn monoid, and $\TR{\cR}$ is $\NTP_2$ then $\TR{\cR}$ is simple.
\end{theorem}
\begin{proof}
Suppose $\TR{\cR}$ is not simple. By Theorem \ref{thm:simple}, we may fix $r,s\in R$ such that $r\leq s$ and $r\p s<2r\p s$. Define a set $A=\{a^{i,j}_1,a^{i,j}_2:i,j<\omega\}$. We define $d$ on $A\times A$ such that 
$$
d(a^{i,j}_m,a^{k,l}_n)=\begin{cases}
r & \text{if $m=n=1$ and $(i,j)\neq (k,l)$}\\
s & \text{if $m=n=2$ and $(i,j)\neq (k,l)$}\\
r\p s & \text{if $m\neq n$, and $i\neq k$ or $j=l$}\\
2r\p s & \text{if $m\neq n$, $i=k$, and $j\neq l$.}
\end{cases}
$$
To verify the triangle inequality for $d$, fix a non-degenerate triangle $\{a^{i,j}_m,a^{k,l}_n,a^{g,h}_p\}$ in $A$. Let $\alpha=d(a^{i,j}_m,a^{k,l}_n)$, $\beta=d(a^{i,j}_m,a^{g,h}_p)$, and $\gamma=d(a^{k,l}_n,a^{g,h}_p)$. We want to show $\alpha\leq\beta\p\gamma$, $\beta\leq\alpha\p\gamma$, and $\gamma\leq\alpha\p\beta$. Without loss of generality, assume $m=n$. If $m=p$ then $\alpha=\beta=\gamma$ and so the desired inequalities hold. If $m\neq p$ then $\alpha\in\{r,s\}$ and $\beta,\gamma\in\{r\p s,2r\p s\}$, so the desired inequalities hold.

We may assume $A\subset\U_\cR$. Define the formula
$$
\vphi(x,y_1,y_2):= d(x,y_1)\leq r\wedge d(x,y_2)\leq s.
$$
We show that $A$ and $\vphi(x,y_1,y_2)$ witness $\TP_2$ for $\TR{\cR}$.

Fix a function $\sigma:\omega\func\omega$ and, given $n<\omega$ and $i\in\{1,2\}$, set $b^n_i=a^{n,\sigma(n)}_i$. Let $B=\{b^n_1,b^n_2:n<\omega\}$. To show that $\{\vphi(x,b^n_1,b^n_2):n<\omega\}$ is consistent, it suffices to show that the function $f:B\func\{r,s\}$, such that $f(b^n_1)=r$ and $f(b^n_2)=s$, describes a one-point metric space extension of $B$ (in general, such functions are called Kat\v{e}tov maps, see \cite{MeUS}). In other words, we must verify the inequalities $|f(u)\m f(v)|\leq d(u,v)\leq f(u)\p f(v)$ for all $u,v\in B$. These are all easily checked.
\begin{comment}
For this, we have:
\begin{enumerate}[\hspace{10pt}$\bullet$]
\item for all $n<\omega$, $|f(b^n_1)\m f(b^n_2)|\leq s$, $f(b^n_1)\p f(b^n_2)=r\p s$ and $d(b^n_1,b^n_2)=r\p s$;
\item for all distinct $m,n<\omega$, $|f(b^m_1)\m f(b^n_1)|=0$, $f(b^m_1)\p f(b^n_1)=2r$, and $d(b^m_1,b^n_1)=r$; 
\item for all distinct $m,n<\omega$, $|f(b^m_2)\m f(b^n_2)|=0$, $f(b^m_2)\p f(b^n_2)=2s$, and $d(b^m_2,b^n_2)=s$;
\item for all distinct $m,n<\omega$, $|f(b^m_1)\m f(b^n_2)|\leq s$, $f(b^m_1)\p f(b^n_2)=r\p_S s$, and $d(b^m_1,b^n_2)=s$.
\end{enumerate}
\end{comment}

Finally, we fix $n<\omega$ and $i<j<\omega$ and show $\vphi(x,a^{n,i}_1,a^{n,i}_2)\wedge\vphi(x,a^{n,j}_1,a^{n,j}_2)$ is inconsistent. Indeed, any realization $c$ would violate the triangle inequality:
\begin{equation*}
d(c,a^{n,i}_2)\p d(c,a^{n,j}_1)\leq r\p s<2r\p s=d(a^{n,i}_2,a^{n,j}_1).\qedhere
\end{equation*}
\end{proof}

\subsection{Examples}

In this section, we give tests for calculating the strong order rank of $\TR{\cR}$. We also simplify the calculation in the case when $\cR$ is archimedean, and give conditions under which the isomorphism type of a finite distance monoid is entirely determined by cardinality and archimedean complexity.  

\begin{definition}
Let $\cR$ be a distance monoid.
\begin{enumerate}
\item Given $r,s\in R$, define $\lceil\frac{r}{s}\rceil=\inf\{n<\omega:r\leq ns\}$, where, by convention, we let $\inf\emptyset=\omega$.
\item Given $t\in R$, define $[t]=\{x\in R:x\asymp t\}$ (see Definition \ref{def:archi}). Define
$$
\arch_\cR(t)=\sup\left\{\left\lceil\textstyle\frac{r}{s}\right\rceil:r,s\in[t]\right\},
$$
where, by convention, we let $\sup\N=\omega$. 
\end{enumerate}
\end{definition}

\begin{proposition}\label{prop:archSO}
Suppose $\cR$ is a countable distance monoid.
\begin{enumerate}[$(a)$]
\item $\arch(\cR)\geq\max\{\arch_\cR(t):t\in R\}$.
\item If $\cR$ is archimedean then, for any $t\in R^{>0}$,
$$
\arch(\cR)=\arch_\cR(t)=\left\lceil\frac{\sup R^{>0}}{\inf R^{>0}}\right\rceil,
$$
where $\sup R^{>0}$ and $\inf R^{>0}$ are calculated in $R^*$.
\end{enumerate}
\end{proposition}
\begin{proof}
Part $(a)$. It suffices to fix $t\in R$ and $r,s\in[t]$, with $s<r$, and show that if $n<\omega$ is such that $ns<r$, then $\arch(\cR)>n$. Since $r,s\in[t]$, there is some $m<\omega$ such that $r\leq ms$, and so we have $ns<ms$. It follows that $ns<(n+1)s$, which gives $\arch(\cR)>n$.

Part $(b)$. Fix $t\in R^{>0}$. Since $\cR$ is archimedean, we have $[t]=R^{>0}$, and so, for the second equality, we want to show $\arch_\cR(t)=m:=\lceil\frac{\sup[t]}{\inf[t]}\rceil\in\N\cup\{\omega\}$. The inequality $\arch_\cR(t)\leq m$ is clear. For the other direction, suppose $n\in\N$ and $m>n$. Then, by upper semicontinuity in $\cR^*$, we have $\inf\{nx:x\in[t]\}<\sup[t]$, and so there are $r,s\in[t]$ such that $ns<r$, i.e. $\lceil\frac{r}{s}\rceil>n$.  

To show the first equality, it suffices by part $(a)$ to show $\arch(\cR)\leq\arch_\cR(t)$. We may assume $\arch_\cR(t)=n<\omega$. In particular, for any $r,s\in R^{>0}$, we have $s\leq nr$. Therefore, for any $r_0,r_1\ldots,r_n\in R$, with $0<r_0\leq r_1\leq\ldots\leq r_n$, we have $r_0\p r_1\p\ldots\p r_n\leq nr_1\leq r_1\p\ldots\p r_n$, as desired.
\end{proof}

\begin{example}$~$
\begin{enumerate}
\item Fix a countable ordered abelian group $\cG=(G,+,\leq,0)$ and a convex subset $I\seq G^{>0}$, which is either closed under addition or contains a maximal element. Let $R=I\cup\{0\}$ and, given $r,s\in R$, set $r\p s=\min\{r+s,\sup I\}$. Then $\cR=(R,\p,\leq,0)$ is Urysohn (see \cite[Proposition 7.3]{CoDM1}). If we further assume that the original group $\cG$ is archimedean, then $\cR$ is archimedean as well, and we have 
$$
\SO(\cU_\cR)=\arch(\cR)=\left\lceil\frac{\sup I}{\inf I}\right\rceil,
$$
where $\sup I$ and $\inf I$ are calculated in $\cR^*$.
\item Using the previous example, we can calculate the model theoretic complexity of many classical examples of Urysohn spaces. In particular, using the notation of Example \ref{allEX}, we have
\begin{enumerate}[$(i)$]
\item $\SO(\cU_\cQ)=\SO(\cU_{\cQ_1})=\SO(\cU_\cN)=\omega$;
\item given $n>0$, $\SO(\cU_{\cR_n})=n$.
\end{enumerate}
\end{enumerate}
\end{example}

Next, we consider the relationship between the archimedean rank and cardinality of a distance monoid $\cR$. To motivate this, we give an example showing that, in Proposition \ref{prop:archSO}$(a)$, the inequality can be strict. Consider $\cS=(\{0,1,2,5,6,7\},+_S,\leq,0)$ (see Example \ref{allEX}$(4)$). The reader may verify that $+_S$ is associative on $S$. Note that $1$ and $5$ are representatives for the two nontrivial archimedean classes in $\cS$, and $\arch_\cS(1)=2=\arch_\cS(5)$. However, $1+_S 5<1+_S 1+_S 5$, and so $\arch(\cS)\geq 3$. In fact, a direction calculation shows $\arch(\cS)=3$. 

The previous discussion shows that, given a distance monoid $\cR$, if $\arch(\cR)\geq n$ then we cannot always expect to have some $t\in R$ with $\arch_\cR(t)\geq n$. On the other hand, we do have the following property.

\begin{proposition}\label{prop:SOclasssize}
Suppose $\cR$ is a distance monoid. If $n<\omega$ and $\arch(\cR)\geq n$ then there is some $t\in R^{>0}$ such that $|[t]|\geq n$.
\end{proposition}
\begin{proof}
Suppose $\arch(\cR)\geq n$. We may clearly assume $n\geq 2$. Fix $r_1,\ldots,r_n\in R$, such that $r_1\leq\ldots\leq r_n$ and $r_2\p\ldots\p r_n<r_1\p\ldots\p r_n$. Given $1\leq i\leq n$, let $s_i=r_i\p\ldots\p r_n$. Since $r_1\leq \ldots\leq r_n$, we have $s_i\in[r_n]$ for all $i$. We prove, by induction on $i$, that $s_{i+1}<s_i$. The base case $s_2<s_1$ is given, so assume $s_{i+1}<s_i$. Suppose, for a contradiction, that $s_{i+1}\leq s_{i+2}$. Then
$$
s_i=r_i\p s_{i+1}\leq r_i\p s_{i+2}\leq r_{i+1}\p s_{i+2}=s_{i+1},
$$
which contradicts the induction hypothesis. Altogether, we have $|[r_n]|\geq n$.
\end{proof}

Combining this result with Corollary \ref{cor:NFSOP} and Theorem \ref{thm:SOR}, we obtain the following numeric upper bound for the strong order rank of $\TR{\cR}$.

\begin{corollary}
If $\cR$ is a Urysohn monoid then $\SO(\cU_\cR)\leq|R^{>0}|$.
\end{corollary}

For the final result of this section, we consider a fixed integer $n>0$. Note that any finite distance monoid is Urysohn (by general \Fraisse\ theory or Theorem \ref{thm:QE}). We have shown that if $\cR$ is a distance monoid, with $|R^{>0}|=n$, then $1\leq \SO(\cU_\cR)\leq n$. The next result addresses the extreme cases.

\begin{theorem}\label{thm:unique}
Fix $n>0$ and suppose $\cR$ is a distance monoid, with $|R^{>0}|=n$.
\begin{enumerate}[$(a)$]
\item $\arch(\cR)=1$ if and only if $\cR\cong(\{0,1,\ldots,n\},\max,\leq,0)$.
\item $\arch(\cR)=n$ if and only if $\cR\cong\cR_n=(\{0,1,\ldots,n\},+_n,\leq,0)$.
\end{enumerate}
\end{theorem}
\begin{proof}
Part $(a)$. We have already shown that $\arch(\cR)=1$ if and only if $\cR$ is ultrametric, in which case the result follows. Indeed, if $\cR$ is ultrametric, then $\cR\cong (S,\max,\leq,0)$ for any linear order $(S,\leq,0)$ with least element $0$ and $n$ nonzero elements.

Part $(b)$. We have already observed that $\arch(\cR_n)=n$, so it suffices to assume $\arch(\cR)=n$ and show $\cR\cong \cR_n$. If $\arch(\cR)=n$ then, by Proposition \ref{prop:SOclasssize}, there is $t\in R^{>0}$, with $|[t]|\geq n$, and so $R^{>0}=[t]$. Therefore $\cR$ is archimedean and $\arch_\cR(t)=n$. If $r=\min R^{>0}$ and $s=\max R^{>0}$ then we must have $(n-1)r<s=nr$, and so $R^{>0}=\{r,2r,\ldots,nr\}$. From this, we clearly have $\cR\cong\cR_n$. 
\end{proof}

\section{Imaginaries and Hyperimaginaries}\label{sec:EHI}

In this section, we give some results concerning the question that originally motivated Casanovas and Wagner \cite{CaWa} to consider the Urysohn spaces $\UR{\cR_n}$. In particular, if we replace $\cR_n$ with $\cS_n=(S_n,+_1,\leq,0)$, where $S_n=\{0,\frac{1}{n},\frac{2}{n},\ldots,1\}$, then, as $\cR_n\cong\cS_n$, we can essentially think of $\UR{\cR_n}$ and $\UR{\cS_n}$ as ``isomorphic" Urysohn spaces. The advantage of working with $\cS_n$ is that we can coherently define an $\caL_{\Q\cap[0,1]}$-theory $T^*_\infty=\bigcup_{n<\omega}\TR{\cS_n}$. In \cite{CaWa}, Casanovas and Wagner show that $T^*_\infty$ does not eliminate hyperimaginaries. In \cite[Proposition 7.5]{CoDM1}, we verified that $T^*_\infty=\TR{\cQ_1}$. This is not observed in \cite{CaWa}, although the authors do describe the non-eliminable hyperimaginaries as resulting from infinitesimal distance. We will refine and generalize the results of \cite{CaWa} for arbitrary Urysohn monoids, in order to obtain necessary conditions for elimination of hyperimaginaries in $\TR{\cR}$. Along the way, we also characterize weak elimination of imaginaries. 

 Given a complete theory $T$, a monster model $\M$, a tuple $\abar\in\M$, and a 0-type-definable equivalence relation $E(\xbar,\ybar)$, with $\ell(\abar)=\ell(\xbar)$, we let $[\abar]_E$ denote the $E$-equivalence class of $\abar$, and use $\abar_E$ to denote the hyperimaginary determined by $[\abar]_E$. If $\ell(\xbar)$ is finite and $E$ is $0$-definable, then $\abar_E$ is an \emph{imaginary}. If $\sigma\in\Aut(\M)$ and $\abar_E\in\M^{\heq}$ then $\sigma(\abar_E)=\abar_E$ if $E(\abar,\sigma(\abar))$ holds. If $e,e'\in\M^{\eq}$ then $e'\in\eacl(e)$ if $\{\sigma(e'):\sigma\in\Aut(\M/e)\}$ is finite, and $e'\in\edcl(e)$ if $\sigma(e')=e'$ for all $\sigma\in\Aut(\M/e)$.

\begin{definition}\label{def:EHI}
$~$
\begin{enumerate}
\item $T$ has \textbf{elimination of hyperimaginaries} if, for any $e\in\M^{\heq}$, there is a small set $A\subset\M^{\eq}$ such that $\Aut(\M/e)=\Aut(\M/A)$.
\item Given an imaginary $e\in\M^{\eq}$, a \textbf{canonical parameter} (resp. \textbf{weak canonical parameter}) for $e$ is a finite real tuple $\cbar\in\M$ such that $\cbar\in\edcl(e)$ (resp. $\cbar\in\eacl(e)$) and $e\in\edcl(\cbar)$. 
\item $T$ has \textbf{(weak) elimination of imaginaries} if every imaginary has a (weak) canonical parameter.
\end{enumerate}
\end{definition}

We will use following characterization of elimination of hyperimaginaries (see \cite[Proposition 18.2]{Cabook}). 

\begin{proposition}\label{prop:ehi}
The following are equivalent.
\begin{enumerate}[$(i)$]
\item $T$ has elimination of hyperimaginaries.
\item Let $E(\xbar,\ybar)$ be a $0$-type-definable equivalence relation, with $\xbar=(x_i)_{i<\mu}$, and fix a real tuple $\abar=(a_i)_{i<\mu}$. Then there is a sequence $(E_i(\xbar^i,\xbar^i))_{i<\lambda}$ of $0$-definable $n_i$-ary equivalence relations, with $\xbar^i=(x^i_{j_1},\ldots,x^i_{j_{n_i}})$ and $\ybar^i=(y^i_{j_1},\ldots,y^i_{j_{n_i}})$ for some $j_1<\ldots<j_{n_i}<\mu$, such that, for all $\bbar,\bbar'\models\tp(\abar)$, $E(\bbar,\bbar')$ holds if and only if, for all $i<\lambda$, $E_i(\bbar,\bbar')$ holds.
\end{enumerate}
\end{proposition}

We first verify that, if $\cR$ is a nontrivial Urysohn monoid, then $\TR{\cR}$  does not have elimination of imaginaries. Specifically, we will show that, as is often the case with \Fraisse\ limits in symmetric relational languages, finite imaginaries are not eliminated.  

\begin{lemma}\label{prop:ehigen}
Let $\M$ be a monster model of a complete first-order theory $T$. Assume that $\acl(C)=C$ for all $C\subset\M$. Fix $n>1$. Given $\abar=(a_1,\ldots,a_n)\in\M^n$ and $f\in\Sym(1,\ldots,n)$, let $\abar_f=(a_{f(1)},\ldots,a_{f(n)})$. Let $E_n$ be the $0$-definable equivalence relation on $\M^n$ such that, given $\abar,\bbar\in\M^n$,
$$
E_n(\abar,\bbar)\miff \bbar=\abar_f\text{ for some $f\in\Sym(1,\ldots,n)$}.
$$
Suppose $\abar\in\M^n$ is a tuple of pairwise distinct elements such that $\abar_f\equiv\abar$ for all $f\in\Sym(1,\ldots,n)$. Then $\abar_{E_n}$ does not have a canonical parameter.
\end{lemma}
\begin{proof}
Fix $n>1$ and $E_n$ as in the statement. Let $\abar\in\M^n$ be a tuple of distinct elements such that $\abar_f\equiv\abar$ for all $f\in\Sym(1,\ldots,n)$. Let $e=\abar_{E_n}$. For any $f\in\Sym(1,\ldots,n)$, we may fix $\sigma_f\in\Aut(\M)$ such that $\sigma_f(\abar)=\abar_f$. 

Suppose, toward a contradiction, that $\cbar$ is a canonical parameter for $e$. Then $\cbar\in\dcl^{\eq}(e)$ and so, since $\sigma_f(e)=e$ for all $f$, we have $\sigma_f(\cbar)=\cbar$ for all $f$. 

\noindent\textit{Case 1}: $\cbar$ contains $a_i$ for some $1\leq i\leq n$.

Since $n>1$ we may fix $j\neq i$ and let $f$ be a permutation of $\{1,\ldots,n\}$ such that $f(i)=j$. Then $\sigma_f(\cbar)=\cbar$ implies $a_i=a_j$, which is a contradiction.

\noindent\textit{Case 2}: $\cbar$ is disjoint from $\abar$. 

Since $\acl(\cbar)=\cbar$, it follows that $\tp(\abar/\cbar)$ has infinitely many realizations, and so we may fix $\abar'\equiv_{\cbar} \abar$ such that $\abar'\backslash\abar\neq\emptyset$. If $\sigma\in\Aut(\M/\cbar)$ is such that $\sigma(\abar)=\abar'$ then, by assumption, $\sigma(e)=e$ and so $E_n(\abar,\abar')$ holds, which is a contradiction. 
\end{proof}

\begin{corollary}\label{cor:EI}
If $\cR$ is a nontrivial Urysohn monoid then $\TR{\cR}$ does not have elimination of imaginaries.
\end{corollary}
\begin{proof}
We verify that $\TR{\cR}$ satisfies the conditions of Lemma \ref{prop:ehigen}. First, the fact that $\acl(C)=C$ for all $C\subset\MU{\cR}$ follows from quantifier elimination and that the \Fraisse\ class $\cK_\cR$ has disjoint amalgamation. Next, given $n>0$, if we fix some $r\in R^{>0}$ then there is $\abar=(a_1,\ldots,a_n)\in\MU{\cR}$ such that $d(a_i,a_j)=r$ for all $i\neq j$. In particular, $\abar_f
\equiv \abar$ for all $f\in\Sym(1,\ldots,n)$. 
\end{proof}

The next goal is to characterize weak elimination of imaginaries for $\TR{\cR}$. We adopt the terminology that, given a complete theory $T$, a monster model $\M$ of $T$, $C\subset\M$, and $p\in S_n(C)$, a \textit{$C$-type-definable equivalence relation on $p$} is a $C$-type-definable relation $E(\xbar,\ybar)$ on $\M^n$, with $\ell(\xbar)=\ell(\ybar)=n$, such that $E$ is an equivalence relation on $\{\abar\in\M^n:\abar\models p\}$. If $E$ is $C$-definable, then we say $E$ is a \textit{$C$-definable equivalence relation on $p$}. We will use the following result of Kruckman \cite[Proposition 5.5.3]{Krthesis}.

\begin{proposition}\label{trivEQ}
Let $T$ be a complete theory and $\M$ a monster model of $T$. Suppose that $T$ satisfies the following properties:
\begin{enumerate}[$(i)$]
\item $\acl(C)=C$ for all $C\subset\M$;
\item for all finite $C\subset\M$ and $p\in S_1(C)$, any $C$-definable equivalence relation $E(x,y)$ on $p$ is either equality (i.e., for all $a,b\models p$, $E(a,b)$ if and only if $a=b$) or trivial (i.e., $E(a,b)$ for all $a,b\models p$).
\end{enumerate}
Then, for all finite $C\subset\M$ and $p\in S_n(C)$, any $C$-definable equivalence relation $E(\xbar,\ybar)$ on $p$ is definable in the language of equality. Precisely, there is a subset $I\seq\{1,\ldots,n\}$ and a subgroup $P$ of $\Sym(I)$ such that, for $\abar,\bbar\models p$, 
$$
E(\abar,\bbar)\miff \bigvee_{f\in P}\bigwedge_{i\in I}a_i=b_{f(i)}.
$$
\end{proposition}

The following is an easy consequence of the previous result.

\begin{corollary}\label{cor:trivEQ}
If $T$ is a complete theory satisfiying conditions $(i)$ and $(ii)$ of Proposition \ref{trivEQ}, then $T$ has weak elimination of imaginaries.
\end{corollary}
\begin{proof}
Let $E$ be a $0$-definable equivalence relation on $\M^n$, and fix $\abar\in \M^n$. Let $I\seq\{1,\ldots,n\}$ and $P$ be as in the conclusion of Proposition \ref{trivEQ}, where $C=\emptyset$ and $p=\tp(\abar)$. Let $\cbar=(a_i:i\in I)$. We show $\cbar$ is a weak canonical parameter for $\abar_E$. First, for any $\sigma\in\Aut(\M/\abar_E)$, we have $\abar,\sigma(\abar)\models p$ and $E(\abar,\sigma(\abar))$, and so $\sigma(\cbar)\in\{(c_{f(i)})_{i\in I}:f\in P\}$. Since $P$ is finite, it follows that $\cbar\in\eacl(\abar_E)$. Second, if $\sigma\in\Aut(\M/\cbar)$ then $\abar,\sigma(\abar)\models p$ and $\bigwedge_{i\in I}a_i=\sigma(a_i)$.  Since the identity is in $P$, we have $E(\abar,\sigma(\abar))$, and so $\sigma(\abar_E)=\abar_E$. 
\end{proof}

For the rest of the section, we fix a Urysohn monoid $\cR$.

\begin{definition}
Fix $C\subset\MU{\cR}$ and $p\in S_1(C)$. 
\begin{enumerate}
\item Define $\alpha(p)=d(a,C)$, where $a\models p$ (this does not depend on choice of $a$).
\item Suppose $E(x,y)$ is a $C$-type-definable equivalence relation on $p$. Define $\Gamma(E,p)\seq R^*$ such that $\alpha\in\Gamma(E,p)$ if and only if there are $a,b\models p$ such that $E(a,b)$ and $d(a,b)=\alpha$. Let $\alpha(E,p)=\sup\Gamma(E,p)$.
\end{enumerate}
\end{definition}

\begin{lemma}\label{lem:WEI}
Fix $C\subset\MU{\cR}$ and $p\in S_1(C)$. Suppose $E(x,y)$ is a $C$-type-definable equivalence relation on $p$.
\begin{enumerate}[$(a)$]
\item If $\alpha\in \Gamma(E,p)$ and $\beta\leq\min\{2\alpha,2\alpha(p)\}$, then $\beta\in\Gamma(E,p)$.
\item $\Gamma(E,p)$ is closed downwards and contains $\alpha(E,p)$.
\item If $a,b\models p$, then $E(a,b)$ if and only if $d(a,b)\leq\alpha(E,p)$.
\item If $\alpha(E,p)<2\alpha(p)$ then $\alpha(E,p)\in\eq(\cR^*)$.
\item If $E(x,y)$ is $C$-definable and $\alpha(E,p)<2\alpha(p)$ then $\alpha(E,p)\in \eq(\cR)$.
\end{enumerate}
\end{lemma}
\begin{proof}
Part $(a)$. Given such $\alpha,\beta$, fix $a,b\models p$ such that $E(a,b)$ holds and $d(a,b)=\alpha$. Since $\beta\leq\min\{2\alpha,2\alpha(p)\}$, there is $b'\equiv_{Ca}b$, with $d(b,b')=\beta$. Then $E(a,b)$ and $E(a,b')$ implies $E(b,b')$, and so $\beta\in \Gamma(E,p)$, as desired.

Part $(b)$. The first claim follows from part $(a)$ and the fact that $\alpha\leq 2\alpha(p)$ for any $\alpha\in\Gamma(E,p)$. For the second claim, first note that if $\alpha(E,p)$ has an immediate predecessor in $R^*$, then we must have $\alpha(E,p)\in\Gamma(E,p)$. Otherwise, by definition of $\alpha(E,p)$, the type
\begin{multline*}
p(x)\cup p(y)\cup E(x,y)\\
\cup\{d(x,y)\leq r:r\in R,~\alpha(E,p)\leq r\}\cup\{d(x,y)>r:r\in R,~\alpha(E,p)>r\}
\end{multline*}
is finitely satisfiable, and so $\alpha(E,p)\in\Gamma(E,p)$.

Part $(c)$. Fix $a,b\models p$. If $E(a,b)$ holds then $d(a,b)\leq\alpha(E,p)$ by definition. Conversely, suppose $d(a,b)=\beta\leq\alpha(E,p)$. Then $\beta\in\Gamma(E,p)$ by part $(b)$. Therefore, there are $a',b'\models p$ such that $E(a',b')$ holds and $d(a',b')=\beta$. Then we have $ab\equiv_C a'b'$, and so $E(a,b)$ holds.  

Part $(d)$. By parts $(a)$ and $(b)$, we have $\min\{2\alpha(E,p),2\alpha(p)\}\leq \alpha(E,p)$. Therefore, if $\alpha(E,p)<2\alpha(p)$ then it follows that $\alpha(E,p)=2\alpha(E,p)$.

Part $(e)$. Assume $E(x,y)$ is $C$-definable and $\alpha(E,p)<2\alpha(p)$. By part $(d)$, it is enough to show $\alpha(E,p)\in R$. We may write,
$$
E(x,y):=\bigvee_{i=1}^n\bigg(\vphi_i(x)\wedge\theta_i(y)\wedge d(x,y)\in(r_i,s_i]\bigg),
$$
where $r_i,s_i\in R$, and $\vphi_i(x),\theta_i(y)$ are formulas with parameters from $C$ (for convenience, interpret $d(x,y)\in(0,0]$ as the formula $x=y$). We may assume $\vphi_i(x),\theta_i(x)\in p$ for all $i\leq n$. By part $(c)$, we may also assume $r_i<\alpha(E,p)$. Without loss of generality, assume $s_1=\max\{s_1,\ldots,s_n\}$. To finish the proof, we show $\alpha(E,p)=s_1$.

First, suppose $s_1<\alpha(E,p)$. By definition, there are $a,b\models p$ such that $E(a,b)$ holds and $d(a,b)>s_1$, which contradicts the explicit expression for $E(x,y)$. Finally, suppose $\alpha(E,p)<s_1$. Let $\beta=\min\{s_1,2\alpha(p)\}$. Then there are $a,b\models p$ such that $d(a,b)=\beta$. Since $r_1<\beta\leq s_1$, it follows that $E(a,b)$ holds, contradicting $\alpha(E,p)<\beta$. 
\end{proof}

\begin{corollary}\label{cor:0deq}
Suppose $\cR$ is a Urysohn monoid and $E(x,y)$ is a $0$-definable nontrivial equivalence relation on $\U_\cR$. Then $E(x,y)$ is equivalent to $d(x,y)\leq r$ for some $r\in\eq(\cR)$.
\end{corollary}
\begin{proof}
We apply Lemma \ref{lem:WEI} with $C=\emptyset$ and $p=\{x=x\}$. Then $\alpha(p)=\sup R^*$ and $\alpha(E,p)<\sup R^*$ since $E$ is nontrivial. Therefore $\alpha(E,p)\in \eq(\cR)$.  
\end{proof}

\begin{theorem}\label{thm:WEI}
If $\cR$ is a Urysohn monoid then $\TR{\cR}$ has weak elimination of imaginaries if and only if $\eq^<(\cR)=\{0\}$.
\end{theorem}
\begin{proof}
First, suppose $r\in\eq^<(\cR)$ is nonzero. Let $E_r(x,y)$ denote the equivalence relation $d(x,y)\leq r$. Fix $a\in\MU{\cR}$, and let $e=a_{E_r}$ and $X=[a]_{E_r}$. Note that $0<r<\sup R^*$ implies $X$ is infinite, but not all of $\U_\cR$. We fix a finite real tuple $\cbar$ and show that $\cbar$ is not a weak canonical parameter for $e$. 

\noindent\textit{Case 1}: $\cbar$ is the empty tuple.

Then any automorphism of $\U_\cR$ fixes $X$ setwise. Since there is a unique $1$-type over $\emptyset$, it follows that $X=\U_\cR$, which is a contradiction.

\noindent\textit{Case 2}: There is some $c\in\cbar\cap X$. 

For any $b\in X$, we may fix $\sigma_ b\in\Aut(\MU{\cR})$ such that $\sigma_b(c)=b$. Then $\sigma_b\in\Aut(\MU{\cR}/e)$, and we have shown that any element of $X$ is in the orbit of $c$ under $\Aut(\MU{\cR}/e)$. Since $X$ is infinite, it follows that $c\not\in\eacl(e)$. 

\noindent\textit{Case 3}: $\cbar$ is nonempty and $\cbar\cap X=\emptyset$. 

Let $\alpha=\min\{d(a,c):c\in\cbar\}$. Then $r<\alpha$, by assumption of this case. Moreover, we may find $a'\in\MU{\cR}$ such that $a'\equiv_{\cbar} a$ and $d(a,a')=\alpha$. If $\sigma\in\Aut(\MU{\cR}/\cbar)$ is such that $\sigma(a)=a'$ then, as $\alpha>r$, we have $\sigma(e)\neq e$. Therefore $e\not\in\edcl(\cbar)$. 

Conversely, suppose $\eq^<(\cR)=\{0\}$. We show that $\TR{\cR}$ satisfies conditions $(i)$ and $(ii)$ of Proposition \ref{trivEQ}. We have already observed condition $(i)$ in the proof of Corollary \ref{cor:EI}. For condition $(ii)$, fix a finite set $C\subset\MU{\cR}$, $p\in S_1(C)$, and a $C$-definable equivalence relation $E(x,y)$ on $p$. By Lemma \ref{lem:WEI}$(c)$, we have that, for any $a,b\models p$, $E(a,b)$ holds if and only if $d(a,b)\leq\alpha(E,p)$. If $\alpha(E,p)=2\alpha(p)$, then it follows that $E(a,b)$ holds for all $a,b\models p$, and so $E(x,y)$ is trivial on $p$. On the other hand, if $\alpha(E,p)<2\alpha(p)$ then, by Lemma \ref{lem:WEI}$(e)$, along with the assumption that $\eq^<(\cR)=\{0\}$, it follows that $\alpha(E,p)=0$. In particular, $E(x,y)$ coincides with equality on $p$.
\end{proof}

\begin{corollary}
Suppose $\cR$ is a Urysohn monoid, and assume that $\SO(\cU_\cR)<\omega$ (i.e. $\arch(\cR)<\omega$). Then $\TR{\cR}$ has weak elimination of imaginaries if and only if $\cR$ is archimedean.
\end{corollary}
\begin{proof}
First, note that in general, $\cR$ archimedean implies $\eq^<(\cR)=\{0\}$. Conversely, let $\arch(\cR)=n<\omega$. Then $nr\in\eq(\cR)$ for any $r\in R$. Therefore, $\eq^<(\cR)=\{0\}$ implies $nr=\sup R$ for all $r\in R^{>0}$, and so $\cR$ is clearly archimedean. 
\end{proof}

\begin{definition}
Define 
$$
\heq(\cR)=\{\alpha\in\eq(\cR^*):\alpha<\inf\{r\in\eq(\cR):\alpha\leq r\}\}.
$$
\end{definition}

\begin{theorem}\label{thm:ehi}
Suppose $\cR$ is a nontrivial Urysohn monoid. If $\heq(\cR)\neq\emptyset$ then $\TR{\cR}$ does not have elimination of hyperimaginaries.
\end{theorem}
\begin{proof}
Suppose $\TR{\cR}$ eliminates hyperimaginaries and fix $\alpha\in\eq(\cR^*)$. We want to show $\alpha=\inf\{r\in\eq(\cR):\alpha\leq r\}$, and we may clearly assume $\alpha<\sup R^*$. Fix a singleton $a\in\MU{\cR}$. Recall that there is a unique $1$-type in $S_1^{\cU_\cR}(\emptyset)$. Since $d(x,y)\leq\alpha$ is a $0$-type-definable equivalence relation, it follows from Proposition \ref{prop:ehi} that there is a sequence $(E_i(x,y))_{i<\lambda}$ of $0$-definable unary equivalence relations such that for any $b,b'\in\M$, $d(b,b')\leq\alpha$ if and only if, for all $i<\lambda$, $E_i(b,b')$ holds. Since $\alpha<\sup R^*$, we may clearly assume that no $E_i$ is trivial on $\M$. By Corollary \ref{cor:0deq}, there are $r_i\in \eq(R)$, for $i<\lambda$, such that $E_i(x,y)$ is equivalent to $d(x,y)\leq r_i$. Therefore, we must have $\alpha=\inf\{r_i:i<\lambda\}$, which gives the desired result.
\end{proof}

Returning to the results of \cite{CaWa}, consider the distance monoid $\cQ_1$. We have the element $0^+\in\cQ^*_1$ such that $0<0^+<r$ for all $r\in\Q\cap(0,1]$. By upper semicontinuity in $\cQ_1^*$, we have $0^+\in\eq(\cQ^*_1)$. Note also that $\eq(\cQ_1)=\{0,1\}$, and so $0^+\in\heq(\cQ_1)$. Therefore, failure of elimination of hyperimaginaries for $\TR{\cQ_1}$ is a special case of the previous result. Note also that $0^+\in\heq(\cQ)$ and so $\TR{\cQ}$ also fails elimination of hyperimaginaries. 

The converse of the previous result is an interesting problem.

\begin{conjecture}\label{conj:ehi}
Suppose $\cR$ is a nontrivial Urysohn monoid. Then $\TR{\cR}$ has elimination of hyperimaginaries if and only if $\heq(\cR)=\emptyset$.
\end{conjecture}

Regarding consequences of this conjecture, we first make the following observation.

\begin{proposition}\label{prop:conj}
If $\cR$ is a countable distance monoid and $\arch(\cR)<\omega$ then $\heq(\cR)=\emptyset$.
\end{proposition}
\begin{proof}
Suppose $\alpha\in\heq(\cR)$. Then there is $\beta\in R^*$ such that $\alpha<\beta$ and, for all $r\in R$, if $\alpha<r<\beta$ then $r<r\p r$. Fix $n>0$. Then $n\alpha=\alpha<\beta$ so, by upper semicontinuity in $\cR^*$, there is some $t\in R$ such that $\alpha<t$ and $nt<\beta$. Then $nt<2nt$, which implies $\arch(\cR)>n$. 
\end{proof}

The purpose of Casanovas and Wagner's work in \cite{CaWa} was to demonstrate the existence of a theory without the strict order property that does not eliminate hyperimaginaries. Our previous work sharpens this upper bound of complexity to \textit{without the finitary strong order property}. On the other hand, if Conjecture \ref{conj:ehi} is true then, combined with Proposition \ref{prop:conj}, we would conclude that generalized Urysohn spaces provide no further assistance in decreasing the complexity of this upper bound. In other words, a consequence of our conjecture is that finite strong order rank for $\TR{\cR}$ implies elimination of hyperimaginaries. An outrageous, but nonetheless open, conjecture could be obtained from this statement by replacing $\TR{\cR}$ with an arbitrary complete theory $T$. Concerning the converse of this statement, note that, if Conjecture \ref{conj:ehi} holds, then $\TR{\cN}$ would eliminate hyperimaginaries, while $\SO(\cU_\cN)=\omega$. As a side note, $\TR{\cN}$ at least eliminates finitary hyperimaginaries, since $\cN^*$ countable implies $S_n(\Th(\cU_\cN))$ countable for all $n<\omega$ (see \cite[Theorem 18.14]{Cabook}).

\begin{remark}\label{rem:thorn}
In \cite{CoStr}, we use the characterization of weak elimination of imaginaries for $\TR{\cR}$ to provide counterexamples to the following question of Adler \cite{Adgeo}. Given a first order theory $T$, one defines a \emph{strict independence relation} for $T^{\eq}$ to be a ternary relation on $\M^{\eq}$ satisfying a certain list of axioms (see \cite{Adgeo}). If $T^{\eq}$ has such a relation, then $T$ is called \textit{rosy} and, moreover, there is a weakest strict independence relation for $T^{\eq}$ called \textit{thorn-forking}. In \cite[Question 1.7]{Adgeo}, Adler asks if the axioms of strict independence relations \textit{characterize} thorn-forking in $T^{\eq}$ when $T$ is rosy. In other words, is there a theory $T$ such that $T^{\eq}$ has more than one strict independence relation? 

Let $\cR$ be a Urysohn monoid and consider the following relation $\ind^{\infty}$ on $\U_\cR$:
$$
A\textstyle\ind^{\infty}_C B\miff \text{$A\cap B\seq C$ and if $d(a,C)=\sup R^*$ then $d(a,BC)=\sup R^*$.}
$$
In \cite{CoStr}, we show that if $\cR$ has no maximal element, then $\ind^{\infty}$ is a strict independence relation for $\TR{\cR}$, which is distinct from \textit{algebraic independence}: $A\ind^a_C B$ if and only if $A\cap B\seq C$. In the case that $\TR{\cR}$ has weak elimination of imaginaries, we further show that $\ind^{\infty}$ and $\ind^a$ ``lift" to distinct strict independence relations for $\TR{\cR}^{\eq}$. For example, this applies whenever $\cR$ is the nonnegative part of a countable ordered abelian group (e.g. $\cN$ or $\cQ$). 
\end{remark}

\appendix

\section{Forking and Dividing in Generalized Urysohn Spaces}\label{app:FD}

\renewcommand*{\thesection}{\Alph{section}}

In this appendix, we briefly discuss the proof of Theorem \ref{thm:forks}, which is essentially a direct translation of the analogous result for the complete Urysohn sphere in continuous logic \cite{CoTe} (joint work with Caroline Terry). A full proof of Theorem \ref{thm:forks} can be found in the author's thesis \cite[Section 3.4]{Cothesis}. First, we summarize a ``guide" for translating from \cite{CoTe}.  

\begin{enumerate}
\item In \cite{CoTe}, we consider a monster model $\U$ of the theory of the complete Urysohn sphere as a metric structure in continuous logic. In this case, $\U$ is a complete, $\kappa^+$-universal and $\kappa$-homogeneous metric space of diameter $1$, where $\kappa$ is the density character of $\U$. Given a tuple $\abar=(a_1,\ldots,a_n)\in\U$ and a subset $C\subset\U$, by quantifier elimination, the type $\tp_{\xbar}(\abar/C)$ is completely determined by 
$$
\{d(x_i,x_j)=d(a_i,a_j):1\leq i,j\leq n\}\cup\{d(x_i,c)=d(a_i,c):1\leq i\leq n,~c\in C\},
$$
where, given variables $x,y$ and $r\in [0,1]$, $d(x,y)=r$ denotes the condition $|d(x,y)-r|=0$. 
\item In this section, we consider a monster model $\U_\cR$ of the theory of the $\cR$-Urysohn space as a relational structure in classical logic. In this case, $\U_\cR$ is a $\kappa^+$-universal and $\kappa$-homogeneous $\cR^*$-metric space, where $\kappa$ is the cardinality of $\U_\cR$. Given a tuple $\abar=(a_1,\ldots,a_n)\in\U_\cR$ and a subset $C\subset\U_\cR$, by quantifier elimination, the type $\tp_{\xbar}(\abar/C)$ is completely determined by 
$$
\bigcup_{1\leq i,j\leq n}d(x_i,x_j)=d(a_i,a_j)\cup\bigcup_{1\leq i\leq n,~c\in C}d(x_i,c)=d(a_i,c),
$$
where, given variables $x,y$ and $\alpha\in R^*$, $d(x,y)=\alpha$ denotes the type
$$
\{d(x,y)\leq r:r\in R,~\alpha\leq r\}\cup\{d(x,y)>r:r\in R,~r<\alpha\}.
$$
\end{enumerate}

For the rest of the section, $\cR$ is a fixed Urysohn monoid. Toward the proof of Theorem \ref{thm:forks}, the first step is to characterize nondividing.  

\begin{theorem}\label{thm:divide}
Given $A,B,C\subset\MU{\cR}$, $A\ind^d_C B$ if and only if for all $b_1,b_2\in B$,
$$
d_{\max}(b_1,b_2/AC)=d_{\max}(b_1,b_2/C)\mand d_{\min}(b_1,b_2/AC)=d_{\min}(b_1,b_2/C).
$$
\end{theorem}

The proof of this theorem can be obtained by directly translating Section 3.2 of \cite{CoTe}, using the translation guide given above. The argument only requires the inequality given by Lemma \ref{lem:dmax} and the fact that $\U_\cR$ is universal and homogeneous as an $\cR^*$-metric space. The key technical step in this argument is the following proposition, which explains the significance of $d_{\max}$ and $d_{\min}$.

\begin{proposition}\label{prop:dmaxmin}
Fix $C\subset\U_\cR$, $b_1,b_2\in\U_\cR$, and $\alpha\in R^*$. The following are equivalent.
\begin{enumerate}[$(i)$]
\item $d_{\min}(b_1,b_2/C)\leq\alpha\leq d_{\max}(b_1,b_2/C)$.
\item There is a $C$-indiscernible sequence $(b^l_1,b^l_2)_{l<\omega}$ such that $(b^0_1,b^0_2)=(b_1,b_2)$ and $d(b^0_1,b^1_2)=\alpha$.
\end{enumerate}
\end{proposition}

Finally, to prove Theorem \ref{thm:forks}, we show that forking and dividing are the same for complete types in $\TR{\cR}$. For this, it suffices to prove the following theorem, which shows that $\ind^d$ satisfies extension in $\TR{\cR}$.

\begin{theorem}\label{forkReduct}
Fix subsets $B,C\subset\MU{\cR}$ and a singleton $b_*\in\MU{\cR}$. For any $A\subset\MU{\cR}$, if $A\ind^d_C B$ then there is $A'\equiv_{BC}A$ such that $A'\ind^d_C Bb_*$.
\end{theorem} 

Once again, this argument closely follows \cite{CoTe}. However, there are a few places where some caution is warranted, and so we give a very terse outline the proof. Fix $B,C\subset\MU{\cR}$ and $b_*\in\MU{\cR}$. 

\begin{definition}
For $a\in\MU{\cR}$, define $U(a):=\inf_{b\in BC}(d(a,b)\p d_{\min}(b_*,b/C))$.
\end{definition}

The motivation for this definition is the observation, which follows easily from Theorem \ref{thm:divide}, that, given $a\in\MU{\cR}$ with $a\ind^d_C B$, if $a'\equiv_{BC} a$ and $a'\ind^d_C Bb_*$, then $d(a',b_*)\leq U(a)$. Toward the proof of Theorem \ref{forkReduct}, the key technical tool is the following result.

\begin{proposition}\label{prop:tech}$~$
\begin{enumerate}[$(a)$]
\item Suppose $a\in\U_\cR$ and $a\ind^d_C B$.
\begin{enumerate}[$(i)$]
\item $d_{\min}(a,b_*/BC)\leq U(a)\leq d_{\max}(a,b_*/BC)$.
\item If $a'\in\MU{\cR}$ is such that $a'\equiv_{BC} a$ and $d(a',b_*)=U(a)$, then $a'\ind^d_C Bb_*$.
\end{enumerate}
\item If $a_1,a_2\in\MU{\cR}$ are such that $a_1a_2\ind^d_C B$ then 
$$
|U(a_1)\m U(a_2)|\leq d(a_1,a_2)\leq U(a_1)\p U(a_2).
$$
\end{enumerate}
\end{proposition}

The proof of this proposition is quite technical, and can be translated directly from \cite[Lemma 3.22]{CoTe} and \cite[Lemma 3.23]{CoTe}. In particular, the claims involve checking a large number of inequalities, which heavily rely on the characterization of $\ind^d$ given by Theorem \ref{thm:divide}, along with Lemma \ref{lem:dmax} and a ``dual" result to Lemma \ref{lem:dmax}, which says that, for all $C\subset\U_\cR$ and $b_1,b_2,b_3\in\U_\cR$, $d_{\min}(b_1,b_3/C) \leq d_{\min}(b_1,b_2/C)\p d_{\min}(b_2,b_3/C)$. The proof of this inequality (which, in \cite{CoTe}, is Lemma 3.16$(b)$) is similar to the proof of Lemma \ref{lem:dmax}, but requires the following applications of upper semicontinuity in $\cR^*$:
\begin{enumerate}[$(i)$]
\item for all $\alpha,\beta,\gamma\in R^*$, $|\alpha\m\gamma|\leq|\alpha\m\beta|\p|\beta\m\gamma|$;
\item for all $\alpha,\beta\in R^*$, $\frac{1}{3}(\alpha\p\beta)\leq\frac{1}{3}\alpha\p\frac{1}{3}\beta$.
\end{enumerate} 
For smoother exposition, Lemmas 3.22 and 3.23 of \cite{CoTe} also use the ``dotminus" operation $r\dotminus s=\max\{r-s,0\}$ on $[0,1]$. The analog in the present setting is the operation $\alpha\odotminus\beta$, which is $0$ when $\alpha<\beta$ and $|\alpha\m\beta|$ when $\beta\leq\alpha$.

We can now prove Theorem \ref{forkReduct}, completing the proof of Theorem \ref{thm:forks}.

\begin{proof}[Proof of Theorem \ref{forkReduct}]
Fix variables $\xbar=(x_a)_{a\in A}$ and define the type
$$
p(\xbar):=\tp_{\xbar}(A/BC)\cup\{d(x_a,b_*)=U(a):a\in A\}.
$$
If $A'$ realizes $p(\xbar)$ then $A'\equiv_{BC} A$ and, by Proposition \ref{prop:tech}$(a)(ii)$, we have $a'\ind^d_C Bb_*$ for all $a'\in A'$. Using Theorem \ref{thm:divide}, we then have $A'\ind^d_C Bb_*$. Therefore, it suffices to show $p(\xbar)$ is consistent, which means verifying the necessary triangle inequalities. The nontrivial triangles to check either have distances $\{d(a,b),d(b,b_*),U(a)\}$ for some $a\in A$ and $b\in BC$, or $\{d(a_1,a_2),U(a_1),U(a_2)\}$ for some $a_1,a_2\in A$. Therefore, the triangle inequality follows, respectively, from parts $(a)(i)$ and $(b)$ of Proposition \ref{prop:tech}. 
\end{proof}

\section*{Acknowledgments} 
This work was done under the supervision of my thesis advisor, David Marker. Many of the results grew from previous joint work with Caroline Terry. I also thank Alex Kruckman for allowing me to use Proposition \ref{trivEQ} before its publication, as well as the referee for several comments and corrections.

\section*{References}

\bibliography{/Users/gabrielconant/Desktop/Math/BibTex/biblio}
\bibliographystyle{amsplain}

\end{document}